\documentclass{amsart}
\usepackage{amssymb}

\newtheorem{theorem}{Theorem}[section]
\newtheorem{lemma}[theorem]{Lemma}

\newtheorem{cor}[theorem]{Corollary}
\theoremstyle{definition}
\newtheorem{definition}[theorem]{Definition}
\newtheorem{example}[theorem]{Example}

\theoremstyle{remark}
\newtheorem{remark}[theorem]{Remark}

\numberwithin{equation}{section}
\begin{document}

\title[summation identities and transformations for hypergeometric series]{summation identities and transformations for hypergeometric series}


\author{Rupam Barman}
\address{Department of Mathematics, Indian Institute of Technology Guwahati, Guwahati-781039, INDIA}
\curraddr{}
\email{rupam@iitg.ernet.in}
\author{Neelam Saikia}
\address{Department of Theoretical Statistics and  Mathematics unit, Indian Statistical Institute, Delhi Centre, New Delhi-110016, INDIA}
\curraddr{}
\email{nlmsaikia1@gmail.com}
\thanks{}

\date{21st September, 2016}
\begin{abstract} We find summation identities and transformations for the McCarthy's $p$-adic hypergeometric series
by evaluating certain Gauss sums which appear while counting points on the family
$$Z_{\lambda}: x_1^d+x_2^d=d\lambda x_1x_2^{d-1}$$
over a finite field $\mathbb{F}_p$.
A. Salerno expresses the number of points over a finite field $\mathbb{F}_p$ on the family $Z_{\lambda}$
in terms of quotients of $p$-adic gamma function under the condition that $d|p-1$. In this paper, we first express
the number of points over a finite field $\mathbb{F}_p$ on the family $Z_{\lambda}$ in terms of McCarthy's $p$-adic hypergeometric series
for any odd prime $p$ not dividing $d(d-1)$, and then deduce two summation identities for the $p$-adic hypergeometric series. We
also find certain transformations and special values of the $p$-adic hypergeometric series.
We finally find a summation identity for the Greene's finite field hypergeometric series.
\end{abstract}
\maketitle
\section{Introduction and statement of results}
It  is  a  well  known  result  that  the  number  of  points over a finite field on the Legendre family of elliptic curves can be
written in terms of a hypergeometric function modulo $p$. In \cite{salerno}, A. Salerno extends this result to a family of monomial
deformations of a diagonal hypersurface. She finds explicit relationships between the number of points and generalized hypergeometric functions
as well as their finite field analogues. Let $X_{\lambda}$ denote the family of monomial deformations of diagonal hypersurfaces
$$X_{\lambda}: x_1^d+x_2^d+\cdots +x_n^d=d\lambda x_1^{h_1}x_2^{h_2}\cdots x_n^{h_n},$$
where $\sum h_i=d$ and $\gcd(d, h_1, \ldots, h_n)=1$. For $\lambda \in \mathbb{Z}$, let $N_{\mathbb{F}_q}(X_{\lambda})$ denote the number of points
on $X_{\lambda}$ in $\mathbb{P}_{\mathbb{F}_q}^{n-1}$, where $\mathbb{F}_q$ is the finite field of $q=p^e$-elements.
Under the condition that $dh_1\cdots h_n| (q-1)$,
A. Salerno \cite[Thm. 4.1]{salerno} expresses $N_{\mathbb{F}_q}(X_{\lambda})-N_{\mathbb{F}_q}(X_{0})$
as a sum of finite field analogues of hypergeometric functions defined by N. Katz \cite{katz}. She studies the special case Dwork family
$X_{\lambda}^d: x_1^d+x_2^d+\cdots +x_d^d=d\lambda x_1x_2\cdots x_d$ when $d=3, 4$.
In \cite{goodson}, H. Goodson gives an expression for the number of points
on the family of Dwork K3 surfaces $X_{\lambda}^4: x_1^4+x_2^4+x_3^4+x_4^4=4\lambda x_1x_2x_3x_4$ over a finite field $\mathbb{F}_q$
in terms of Greene's finite field hypergeometric functions under the condition that $q\equiv 1 \pmod{4}$.
She further gives an expression for the number of points on the family
$X_{\lambda}^4$ in terms of McCarthy's $p$-adic hypergeometric series ${_n}G_{n}[\cdots]$ (defined in Section 2)
under the condition that $p\not \equiv 1\pmod{4}$. Recently, the authors with H. Rahman \cite{BRS} express the number of $\mathbb{F}_q$-points on $X_{\lambda}^d$ in terms
of McCarthy's $p$-adic hypergeometric series when $d$ is any odd prime such that $p\nmid d$ and $q\not\equiv 1 \pmod{d}$, which gives a solution to
a conjecture of H. Goodson \cite{goodson}.
\par The aim of this paper is to find summation identities and transformations
for the McCarthy's $p$-adic hypergeometric series and Greene's finite field hypergeometric series.
In \cite{BS5}, the authors with D. McCarthy
find eight summation identities for the $p$-adic hypergeometric series by counting points on certain hyperelliptic curves over a finite field. Here we apply
similar technique to the 0-dimensional variety $Z_{\lambda}: x_1^d+x_2^d=d\lambda x_1x_2^{d-1}$ and deduce the summation identities. Under the condition
that $d|p-1$, A. Salerno expresses the number of points over a finite field $\mathbb{F}_p$ on the family $Z_{\lambda}$
in terms of quotients of $p$-adic gamma function (for example, see \cite[Lemma 5.4]{salerno}). In the following theorem, we express
the number of points over a finite field $\mathbb{F}_p$ on the family $Z_{\lambda}$ in terms of McCarthy's $p$-adic hypergeometric series
for any odd prime $p$ not dividing $d(d-1)$.
\begin{theorem}\label{point-count}
Let $p$ be an odd prime such that $p\nmid d(d-1)$. If $\lambda \neq 0$, then the number of $\mathbb{F}_p$-points $N_{\mathbb{F}_p}(Z_{\lambda})$ on the
$0$-dimensional variety $Z_\lambda^d: x_1^d+x_2^d=d\lambda x_1x_2^{d-1}$ is given by
\begin{eqnarray}
N_{\mathbb{F}_p}(Z_{\lambda})=1+{_{d-1}G}_{d-1}\left[\begin{array}{cccc}
                                           \frac{1}{d}, & \frac{2}{d}, & \ldots, & \frac{d-1}{d} \vspace{.1cm}\\
                                           0, & \frac{1}{d-1}, & \ldots, & \frac{d-2}{d-1}
                                         \end{array}|\lambda^d(d-1)^{d-1}\right].\nonumber
\end{eqnarray}
\end{theorem}
We evaluate certain Gauss sums which appear while counting points on $Z_{\lambda}$ over $\mathbb{F}_p$ and deduce the following two summation identities.
Let $\phi$ denote the quadratic character on $\mathbb{F}_p$.
\begin{theorem}\label{MT1}
Let $d\geq3$ be odd and $p$ an odd prime such that $p\nmid d(d-1)$. For $x\in\mathbb{F}_p^{\times}$ we have
\begin{eqnarray}
&&\hspace{-.4cm}\sum_{t \in \mathbb{F}_p}\phi(t(t-1))\times\nonumber\\
&&\hspace{-.4cm}{_{d-1}G}_{d-1}\left[\begin{array}{cccccccccc}
                       \frac{1}{d}, \hspace{-.2cm} & \frac{2}{d}, \hspace{-.2cm} & \ldots, \hspace{-.2cm} & \frac{\frac{d-1}{2}-1}{d},
                       \hspace{-.2cm}& \frac{\frac{d-1}{2}}{d}, \hspace{-.2cm}
 & \frac{\frac{d-1}{2}+1}{d}, \hspace{-.2cm} & \ldots, \hspace{-.2cm} & \frac{d-3}{d}, \hspace{-.2cm} & \frac{d-2}{d}, \hspace{-.2cm} & \frac{d-1}{d}
 \vspace{.1cm}\\
 \frac{1}{d-1}, \hspace{-.2cm} & \frac{2}{d-1}, & \hspace{-.2cm}\ldots, \hspace{-.2cm} & \frac{\frac{d-1}{2}-1}{d-1},
 \hspace{-.2cm} & \frac{\frac{d-1}{2}+1}{d-1}, \hspace{-.2cm} & \frac{\frac{d-1}{2}+2}{d-1}, & \hspace{-.2cm} \ldots, \hspace{-.2cm} & \frac{d-2}{d-1}, \hspace{-.2cm}
 & 0, \hspace{-.2cm} &0
                     \end{array}\hspace{-.1cm}|xt\right]\nonumber\\
&&\hspace{-.4cm}=-1-
p\cdot {_{d-1}G}_{d-1}\left[\begin{array}{cccc}
                                           \frac{1}{d}, & \frac{2}{d}, & \ldots, & \frac{d-1}{d} \vspace{.1cm}\\
                                           0, & \frac{1}{d-1}, & \ldots, & \frac{d-2}{d-1}
                                         \end{array}|x
\right].\nonumber
\end{eqnarray}
\end{theorem}
\begin{theorem}\label{MT2}
Let $d>2$ be even and $p$ an odd prime such that $p\nmid d(d-1)$. For $x\in\mathbb{F}_p^{\times}$ we have
\begin{eqnarray}
&&\hspace{-.4cm}\sum_{t \in \mathbb{F}_p}\phi(1-t){_{d-2}G}_{d-2}\left[\begin{array}{ccccccccc}
\frac{1}{d}, \hspace{-.2cm} & \frac{2}{d}, \hspace{-.2cm} & \ldots, \hspace{-.2cm} & \frac{\frac{d}{2}-1}{d}, \hspace{-.2cm}
& \frac{\frac{d}{2}+1}{d}, \hspace{-.2cm} & \frac{\frac{d}{2}+2}{d}, \hspace{-.2cm}
& \ldots, \hspace{-.2cm} & \frac{d-2}{d}, & \frac{d-1}{d}
\vspace{.1cm}\\
\frac{1}{d-1}, \hspace{-.2cm} & \frac{2}{d-1}, \hspace{-.2cm} &\ldots, \hspace{-.2cm} &\frac{\frac{d}{2}-1}{d-1}, \hspace{-.2cm} & \frac{\frac{d}{2}}{d-1},
\hspace{-.2cm} & \frac{\frac{d}{2}+1}{d-1}, \hspace{-.2cm} &\ldots, \hspace{-.2cm} & \frac{d-3}{d-1}, \hspace{-.2cm} & \frac{d-2}{d-1}
\end{array}|xt\right]\nonumber\\
&&=-{_{d-1}G}_{d-1}\left[\begin{array}{cccc}
                                           \frac{1}{d}, & \frac{2}{d}, & \ldots, & \frac{d-1}{d} \vspace{.1cm}\\
                                           0, & \frac{1}{d-1}, & \ldots, & \frac{d-2}{d-1}
                                         \end{array}|x\right].\nonumber
\end{eqnarray}
\end{theorem}
Using the summation identities, we obtain the following two point count formulas for $Z_{\lambda}$.
\begin{cor}
Let $d>2$ be even and $p$ an odd prime such that $p\nmid d(d-1)$. Then
\begin{eqnarray}
 N_{\mathbb{F}_p}(Z_{\lambda})&=1&-\sum_{t \in \mathbb{F}_p}\phi(1-t)\times \nonumber\\
 &&{_{d-2}G}_{d-2}\left[\begin{array}{ccccccccc}
\frac{1}{d}, \hspace{-.2cm} & \frac{2}{d}, \hspace{-.2cm} & \ldots, \hspace{-.2cm} & \frac{\frac{d}{2}-1}{d}, \hspace{-.2cm}
& \frac{\frac{d}{2}+1}{d}, \hspace{-.2cm} & \frac{\frac{d}{2}+2}{d}, \hspace{-.2cm}
& \ldots, \hspace{-.2cm} & \frac{d-2}{d}, & \frac{d-1}{d}
\vspace{.1cm}\\
\frac{1}{d-1}, \hspace{-.2cm} & \frac{2}{d-1}, \hspace{-.2cm} &\ldots, \hspace{-.2cm} &\frac{\frac{d}{2}-1}{d-1}, \hspace{-.2cm} & \frac{\frac{d}{2}}{d-1},
\hspace{-.2cm} & \frac{\frac{d}{2}+1}{d-1}, \hspace{-.2cm} &\ldots, \hspace{-.2cm} & \frac{d-3}{d-1}, \hspace{-.2cm} & \frac{d-2}{d-1}
\end{array}|\alpha t\right],\nonumber
\end{eqnarray}
where $\alpha=\lambda^d(d-1)^{d-1}$.
\end{cor}
\begin{cor}
Let $d\geq3$ be odd and $p$ an odd prime such that $p\nmid d(d-1)$. Then
\begin{eqnarray}
&&pN_{\mathbb{F}_p}(Z_{\lambda})=p-1-\sum_{t \in \mathbb{F}_p}\phi(t(t-1))\times \nonumber\\
&&{_{d-1}G}_{d-1}\left[\begin{array}{cccccccccc}
                       \frac{1}{d}, \hspace{-.2cm} & \frac{2}{d}, \hspace{-.2cm} & \ldots, \hspace{-.2cm} & \frac{\frac{d-1}{2}-1}{d},
                       \hspace{-.2cm}& \frac{\frac{d-1}{2}}{d}, \hspace{-.2cm}
 & \frac{\frac{d-1}{2}+1}{d}, \hspace{-.2cm} & \ldots, \hspace{-.2cm} & \frac{d-3}{d}, \hspace{-.2cm} & \frac{d-2}{d}, \hspace{-.2cm} & \frac{d-1}{d}
 \vspace{.1cm}\\
 \frac{1}{d-1}, \hspace{-.2cm} & \frac{2}{d-1}, & \hspace{-.2cm}\ldots, \hspace{-.2cm} & \frac{\frac{d-1}{2}-1}{d-1},
 \hspace{-.2cm} & \frac{\frac{d-1}{2}+1}{d-1}, \hspace{-.2cm} & \frac{\frac{d-1}{2}+2}{d-1}, & \hspace{-.2cm} \ldots, \hspace{-.2cm} & \frac{d-2}{d-1}, \hspace{-.2cm}
 & 0, \hspace{-.2cm} &0
                     \end{array}\hspace{-.1cm}|\alpha t
\right],\nonumber
\end{eqnarray}
where $\alpha=\lambda^d(d-1)^{d-1}$. Hence,
\begin{eqnarray}
&&\sum_{t \in \mathbb{F}_p}\phi(t(t-1))\times \nonumber\\
&&{_{d-1}G}_{d-1}\left[\begin{array}{cccccccccc}
                       \frac{1}{d}, \hspace{-.2cm} & \frac{2}{d}, \hspace{-.2cm} & \ldots, \hspace{-.2cm} & \frac{\frac{d-1}{2}-1}{d},
                       \hspace{-.2cm}& \frac{\frac{d-1}{2}}{d}, \hspace{-.2cm}
 & \frac{\frac{d-1}{2}+1}{d}, \hspace{-.2cm} & \ldots, \hspace{-.2cm} & \frac{d-3}{d}, \hspace{-.2cm} & \frac{d-2}{d}, \hspace{-.2cm} & \frac{d-1}{d}
 \vspace{.1cm}\\
 \frac{1}{d-1}, \hspace{-.2cm} & \frac{2}{d-1}, & \hspace{-.2cm}\ldots, \hspace{-.2cm} & \frac{\frac{d-1}{2}-1}{d-1},
 \hspace{-.2cm} & \frac{\frac{d-1}{2}+1}{d-1}, \hspace{-.2cm} & \frac{\frac{d-1}{2}+2}{d-1}, & \hspace{-.2cm} \ldots, \hspace{-.2cm} & \frac{d-2}{d-1}, \hspace{-.2cm}
 & 0, \hspace{-.2cm} &0
                     \end{array}\hspace{-.1cm}|\alpha t
\right]\nonumber\\
&& \equiv p-1 \pmod{p}.\nonumber
\end{eqnarray}
\end{cor}
In the following example, we take some values of $d$ to show how our results are applied to particular cases.
\begin{example}
 We put $d=5$ and $d=4$ in Theorem \ref{MT1} and Theorem \ref{MT2}, respectively. Then, for $x\in\mathbb{F}_p^{\times}$,
 we have the following summation identities.
 \begin{eqnarray}
\sum_{t \in \mathbb{F}_p}\phi(t(t-1)){_{4}G}_{4}\left[\begin{array}{cccc}
\frac{1}{5}, & \frac{2}{5}, & \frac{3}{5}, & \frac{4}{5}
\vspace{.1cm} \\
\frac{1}{4}, & \frac{3}{4}, & 0, & 0
\end{array}|xt\right] &=&-1-p\cdot {_{4}G}_{4}\left[\begin{array}{cccc}
\frac{1}{5}, & \frac{2}{5}, & \frac{3}{5}, & \frac{4}{5}
\vspace{.1cm}\\
0, & \frac{1}{4}, & \frac{1}{2}, & \frac{3}{4}
\end{array}|x\right],\nonumber
\\
\sum_{t \in \mathbb{F}_p}\phi(1-t){_{2}G}_{2}\left[\begin{array}{cc}
\frac{1}{4}, & \frac{3}{4}
\vspace{.1cm}\\
\frac{1}{3}, &\frac{2}{3}
\end{array}|xt\right]&=&-{_{3}G}_{3}\left[\begin{array}{ccc}
\frac{1}{4}, & \frac{1}{2}, & \frac{3}{4}
\vspace{.1cm} \\
0, & \frac{1}{3}, & \frac{2}{3}
\end{array}|x\right].\nonumber
\end{eqnarray}
The first identity is valid for $p=3$ and all $p>5$; whereas the second identity is valid for all prime $p>3$.
\end{example}
In \cite{Fuselier-McCarthy}, J. Fuselier and D. McCarthy stablish certain transformations and identities for the $G$-function, 
and use them to prove a supercongruence conjecture of Rodriguez-Villegas between a truncated ${_4}F_3$ classical hypergeometric series
and the $p$-th Fourier coefficients of a weight four modular form, modulo $p^3$. Here, we prove that the $G$-function satisfies the following transformations. 
\begin{theorem}\label{MT-6}
 Let $d\geq 2$ and $p$ an odd prime such that $p\nmid d(d-1)$. For $\lambda\in\mathbb{F}_p^{\times}$ we have
 \begin{eqnarray}
 &&{_{d-1}G}_{d-1}\left[\begin{array}{cccc}
                                           \frac{1}{d}, & \frac{2}{d}, & \ldots, & \frac{d-1}{d} \vspace{.1cm}\\
                                           0, & \frac{1}{d-1}, & \ldots, & \frac{d-2}{d-1}
                                         \end{array}|\lambda
\right]\nonumber\\
&&=\left\{
                                  \begin{array}{ll}
                                    \phi(-\lambda(d-1))\\ \times{_{d-1}}G_{d-1}\left[\begin{array}{ccccccc}
                       \frac{1}{2(d-1)},\hspace{-.2cm} & \frac{3}{2(d-1)}, \hspace{-.2cm}& \ldots, \hspace{-.2cm}& \frac{d-1}{2(d-1)},
                       \hspace{-.2cm}& \frac{d+1}{2(d-1)}, \hspace{-.2cm}& \ldots, \hspace{-.2cm} & \frac{2(d-1)-1}{2(d-1)}\\
                       0, \hspace{-.2cm}& \frac{1}{d}, \hspace{-.2cm}& \ldots, \hspace{-.2cm}& \frac{\frac{d}{2}-1}{d}, \hspace{-.2cm}
                       & \frac{\frac{d}{2}+1}{d}, \hspace{-.2cm}& \ldots, \hspace{-.2cm} & \frac{d-1}{d} \end{array}|\frac{1}{\lambda}
\right]\\
\hfill \hbox{if~ $d$ is even;} \\
                                    \phi(d\lambda){_{d-1}}G_{d-1}\left[\begin{array}{cccccccc}
                      0, \hspace{-.2cm}& \frac{1}{d-1}, \hspace{-.2cm}& \ldots, \hspace{-.2cm}& \frac{d-3}{2(d-1)}, \hspace{-.2cm}& \frac{d-1}{2(d-1)},
                      \hspace{-.2cm}&\ldots, \hspace{-.2cm}& \frac{d-3}{d-1}, \hspace{-.2cm}& \frac{d-2}{d-1} \\
                      \frac{1}{2d}, \hspace{-.2cm}& \frac{3}{2d}, \hspace{-.2cm}& \ldots, \hspace{-.2cm}& \frac{d-2}{2d}, \hspace{-.2cm}&
                      \frac{d+2}{2d}, \hspace{-.2cm}& \ldots, \hspace{-.2cm}& \frac{2d-3}{2d}, \hspace{-.2cm}& \frac{2d-1}{2d}
                    \end{array}
|\frac{1}{\lambda}
\right]\\
\hfill \hbox{if ~~$d\geq 3$ is odd,}
                                  \end{array}
                                \right.\nonumber
 \end{eqnarray}
\end{theorem}
For example, if we put $d=6$, then for all prime $p>5$, we have
\begin{eqnarray}
&&{_{5}G}_{5}\left[\begin{array}{ccccc}
                                           \frac{1}{6}, & \frac{2}{6}, & \frac{3}{6}, & \frac{4}{6}, & \frac{5}{6} \vspace{.1cm}\\
                                           0, & \frac{1}{5}, &\frac{2}{5}, &\frac{3}{5}, & \frac{4}{5}
                                         \end{array}|\lambda
\right]=\phi(-5\lambda){_{5}G}_{5}\left[\begin{array}{ccccc}
                                           \frac{1}{10}, & \frac{3}{10}, & \frac{5}{10},& \frac{7}{10}, & \frac{9}{10} \vspace{.1cm}\\
                                           0, & \frac{1}{6},&\frac{2}{6}, & \frac{4}{6}, & \frac{5}{6}
                                         \end{array}|\frac{1}{\lambda}
\right].\nonumber
\end{eqnarray}
\begin{theorem}\label{MT-5}
For $p>7$ and $p\neq23$ we have
\begin{eqnarray}
{_4G}_4\left[\begin{array}{cccc}
         0, \hspace{-.1cm}& \frac{1}{4}, \hspace{-.1cm}& \frac{1}{2}, \hspace{-.1cm}& \frac{3}{4} \vspace{.1 cm} \\
         \frac{1}{10}, \hspace{-.1cm}& \frac{3}{10}, \hspace{-.1cm}& \frac{7}{10}, \hspace{-.1cm}& \frac{9}{10}
       \end{array}|-\frac{5^5}{4^4}\right]&=&\phi(-1)+\phi(3)+\phi(-1)~{_2G}_2\left[\begin{array}{cc}
                                                                         \frac{1}{3}, & \frac{2}{3} \vspace{.1 cm} \\
                                                                         0, & \frac{1}{2}
                                                                       \end{array}|\frac{4}{27}\right];\nonumber\\
{_4G}_4\left[\begin{array}{cccc}
               \frac{1}{5}, & \frac{2}{5}, & \frac{3}{5}, & \frac{4}{5} \vspace{.1 cm} \\
               0, & \frac{1}{4}, & \frac{1}{2}, & \frac{3}{4}
             \end{array}|-\frac{4^4}{5^5}\right]&=&1+\phi(-3)+{_2G}_2
                                                       \left[\begin{array}{cc}
                                                               0, & \frac{1}{2} \vspace{.1 cm}\\
                                                               \frac{1}{6}, & \frac{5}{6}
                                                             \end{array}|\frac{27}{4}
                                                       \right];\nonumber\\
                                                       &=&1+\phi(-3)+{_2G}_2
                                                       \left[\begin{array}{cc}
                                                               \frac{1}{3}, & \frac{2}{3} \vspace{.1 cm}\\
                                                               0, & \frac{1}{2}
                                                             \end{array}|\frac{4}{27}
                                                       \right].\nonumber
                                                   \end{eqnarray}

\end{theorem}
From Theorem \ref{MT1} and Theorem \ref{MT-5}, we have the following summation identities.
\begin{cor}
For $p>7$ and $p\neq23$ we have
\begin{eqnarray}
&&\sum_{t\in\mathbb{F}_p}\phi(t(t-1))~{_2G}_2\left[\begin{array}{cc}
               \frac{1}{3}, & \frac{2}{3} \vspace{.1 cm} \\
               0, & 0
             \end{array}|\frac{4t}{27}\right]\nonumber\\
             &&\hspace{2cm}=p-1+p\phi(-3)-p\phi(-1){_4G}_4\left[\begin{array}{cccc}
         0, & \frac{1}{4}, & \frac{1}{2}, & \frac{3}{4} \vspace{.1 cm} \\
         \frac{1}{10}, & \frac{3}{10}, & \frac{7}{10}, & \frac{9}{10}
       \end{array}|-\frac{5^5}{4^4}\right];\nonumber\\
&&\sum_{t\in\mathbb{F}_p}\phi(t(t-1)){_4G}_4\left[\begin{array}{cccc}
         \frac{1}{5}, & \frac{2}{5}, & \frac{3}{5}, & \frac{4}{5} \vspace{.1 cm} \\
         0, &  0, & \frac{1}{4}, & \frac{3}{4}
       \end{array}|-\frac{4^4t}{5^5}\right]\nonumber\\
       &&\hspace{2cm}=-1-p-p\phi(-3)-p\cdot{_2G}_2
\left[\begin{array}{cc}
0, & \frac{1}{2} \vspace{.1 cm}\\
\frac{1}{6}, & \frac{5}{6}
\end{array}|\frac{27}{4}
\right];\nonumber\\
&&\hspace{2cm}=-1-p-p\phi(-3)-p\cdot{_2G}_2\left[\begin{array}{cc}
\frac{1}{3}, & \frac{2}{3} \vspace{.1 cm}\\
0, & \frac{1}{2}
\end{array}|\frac{4}{27}
\right].\nonumber
\end{eqnarray}
\end{cor}
\par
Finally, we find a summation identity for the Greene's finite field hypergeometric series. We first recall some definitions to state our results.
Let $q=p^e$ be a power of an odd prime $p$ and $\mathbb{F}_q$
the finite field of $q$ elements. Let $\widehat{\mathbb{F}_q^{\times}}$ be the group of all multiplicative characters $\chi: \mathbb{F}_q^{\times}\rightarrow
\mathbb{C}^{\times}$. We extend the domain of each
$\chi \in \mathbb{F}_q^{\times}$ to $\mathbb{F}_q$ by setting $\chi(0):=0$ including the trivial character $\varepsilon$.
If $A$ and $B$ are two characters on $\mathbb{F}_q$, then ${A \choose B}$ is defined by
\begin{align}
{A \choose B}:=\frac{B(-1)}{q}\sum_{x \in \mathbb{F}_q}A(x)\overline{B}(1-x),\nonumber
\end{align}
where $\overline{B}$ is the character inverse of $B$.
In \cite{greene}, J. Greene introduced the notion of hypergeometric
series over finite fields which are also known as \emph{Gaussian hypergeometric
series}. For any positive integer $n$ and characters $A_0, A_1,\ldots, A_n$ and $B_1, B_2,\ldots, B_n \in \widehat{\mathbb{F}_q^{\times}}$,
the Gaussian hypergeometric series ${_{n+1}}F_n$ is defined to be
\begin{align}
{_{n+1}}F_n\left(\begin{array}{cccc}
                A_0, & A_1, & \ldots, & A_n\\
                 & B_1, & \ldots, & B_n
              \end{array}\mid x \right):=\frac{q}{q-1}\sum_{\chi}{A_0\chi \choose \chi}{A_1\chi \choose B_1\chi}
\cdots {A_n\chi \choose B_n\chi}\chi(x),\nonumber
\end{align}
where the sum is over all multiplicative characters $\chi$ on $\mathbb{F}_q$.
\par The motivation for deriving summation identities for Greene's hypergeometric series is the following summation identity
due to Greene \cite[Theorem 3.13]{greene}. Let $A_0, A_1, \ldots, A_n, B_1, \ldots, B_n$ be multiplicative characters on $\mathbb{F}_q$
and let $x\in \mathbb{F}_q$. Greene proved that
\begin{eqnarray}\label{greene-summation}
 &&\hspace{.5cm}{_{n+1}}F_n\left(\begin{array}{cccc}
                A_0, \hspace{-.2cm} & A_1, \hspace{-.2cm} & \ldots, \hspace{-.2cm} & A_n\\
                 & B_1, \hspace{-.2cm} & \ldots, \hspace{-.2cm} & B_n
              \end{array}\mid x \right)\\
&&=\dfrac{A_nB_n(-1)}{q}\sum_{y\in\mathbb{F}_q}{_{n}}F_{n-1}\left(\begin{array}{cccc}
                A_0, \hspace{-.2cm} & A_1, \hspace{-.2cm} & \ldots, \hspace{-.2cm} & A_{n-1}\\
                 & B_1, \hspace{-.2cm} & \ldots, \hspace{-.2cm} & B_{n-1}
              \end{array}\mid xy \right)A_n(y)\overline{A_n}B_n(1-y).\nonumber
\end{eqnarray}
We first express the number of $\mathbb{F}_q$-points on $Z_{\lambda}$ in terms of Greene's hypergeometric series in the following result.
\begin{theorem}\label{point-count2}
Let $p$ be an odd prime and $q=p^e$ for some $e>0$. Let $d\geq 3$ be odd such that $q\equiv1\pmod{d(d-1)}$.
For $\lambda \neq 0$, the number of $\mathbb{F}_q$-points $N_{\mathbb{F}_q}(Z_{\lambda})$ on the
$0$-dimensional variety $Z_\lambda^d: x_1^d+x_2^d=d\lambda x_1x_2^{d-1}$ is given by
\begin{eqnarray}
&&q\cdot N_{\mathbb{F}_q}(Z_\lambda)=q-1+q^{\frac{d-1}{2}}\sum_{t\in\mathbb{F}_q}\phi(1-t)\nonumber\\
&&\times{_{d-1}F}_{d-2}\left(\begin{array}{cccccccc}
                       \chi^{\frac{d-1}{2}}, \hspace{-.2cm}& \chi, \hspace{-.2cm}& \ldots, \hspace{-.2cm}& \chi^{\frac{d-1}{2}-1},\hspace{-.2cm}
                       & \chi^{\frac{d-1}{2}+1},
\hspace{-.2cm} & \chi^{\frac{d-1}{2}+2}, \hspace{-.2cm} & \ldots, \hspace{-.2cm} & \chi^{d-1} \\
                      ~ & \hspace{-.2cm} \psi, \hspace{-.2cm} & \ldots, \hspace{-.2cm} & \psi^{\frac{d-1}{2}-1}, \hspace{-.2cm}
                      & \varepsilon, \hspace{-.2cm} & \psi^{\frac{d-1}{2}+1}, \hspace{-.2cm} & \ldots, \hspace{-.2cm} & \psi^{d-2}
\end{array}|\frac{t}{\alpha}\right),\nonumber
\end{eqnarray}
where $\chi$ and $\psi$ are characters of order $d$ and $d-1$ respectively, and $\alpha=\lambda^d(d-1)^{d-1}$.
\end{theorem}
Using the above point-count formula, we prove the following summation identity.
Unlike to \eqref{greene-summation} our summation identity contains characters of specific orders.
It would be interesting to know if the identity could be derived from \eqref{greene-summation}.
\begin{theorem}\label{MT-4}
Let $p$ be an odd prime and $q=p^e$ for some $e>0$. Let $d\geq 3$ be odd such that $q\equiv1\pmod{d(d-1)}$.
For $\lambda\in\mathbb{F}_q^{\times}$ we have
\begin{eqnarray}
&&\hspace{-.5cm}\sum_{t\in\mathbb{F}_q}\phi(1-t){_{d-1}F}_{d-2}\left(\begin{array}{cccccccc}
                       \chi^{\frac{d-1}{2}}, \hspace{-.2cm} & \chi, \hspace{-.2cm}& \ldots, \hspace{-.2cm}& \chi^{\frac{d-1}{2}-1}, \hspace{-.2cm}& \chi^{\frac{d-1}{2}+1},
& \hspace{-.2cm}\chi^{\frac{d-1}{2}+2},\hspace{-.2cm} & \ldots, \hspace{-.2cm}& \chi^{d-1} \\
                      ~ & \psi,\hspace{-.2cm} & \ldots, \hspace{-.2cm}& \psi^{\frac{d-1}{2}-1},\hspace{-.2cm} & \varepsilon,\hspace{-.2cm}
                      & \psi^{\frac{d-1}{2}+1}, \hspace{-.2cm}& \ldots,\hspace{-.2cm}
& \psi^{d-2}
                     \end{array}|\lambda t
\right)\nonumber\\
&&\hspace{-.5cm}=\frac{1-\phi(-\lambda)}{q^{\frac{d-1}{2}}}+q\phi(-1){_dF}_{d-1}\left(\begin{array}{ccccccc}
                   \phi,\hspace{-.2cm} & \chi, \hspace{-.2cm} & \ldots, \hspace{-.2cm} & \chi^{\frac{d-1}{2}}, \hspace{-.2cm}
                   & \chi^{\frac{d-1}{2}+1}, \hspace{-.2cm} & \ldots, & \chi^{d-1} \\
                   ~ &  \psi, \hspace{-.2cm} & \ldots, \hspace{-.2cm} & \psi^{\frac{d-1}{2}}, \hspace{-.2cm} & \psi^{\frac{d-1}{2}}, \hspace{-.2cm}
                   & \ldots, \hspace{-.2cm} & \psi^{d-2}
                 \end{array}|\lambda\right),\nonumber
\end{eqnarray}
where $\chi$ and $\psi$ are characters of order $d$ and $d-1$ respectively.
\end{theorem}
If we put $d=3$ in Theorem \ref{MT-4}, then, for $\lambda\neq 0$, we have
\begin{eqnarray}
&&\sum_{t\in\mathbb{F}_q}\phi(1-t){_{2}F}_{1}\left(\begin{array}{cc}
                       \chi_3, & \chi_3^2 \\
                      ~ & \varepsilon
                     \end{array}|\lambda t
\right)\nonumber\\
&&=\frac{1-\phi(-\lambda)}{q}+q\phi(-1){_3F}_{2}\left(\begin{array}{ccc}
                   \phi, & \chi_3, & \chi_3^{2} \\
                   ~ &  \phi, & \phi
                 \end{array}|\lambda\right),\nonumber
\end{eqnarray}
where $\chi_3$ is a charcater of order 3. In particular, if we take $\lambda=-1$, then we have
\begin{eqnarray}
&&\sum_{t\in\mathbb{F}_q}\phi(1+t){_{2}F}_{1}\left(\begin{array}{cc}
                       \chi_3, & \chi_3^2 \\
                      ~ & \varepsilon
                     \end{array}| t
\right)=q\phi(-1){_3F}_{2}\left(\begin{array}{ccc}
                   \phi, & \chi_3, & \chi_3^{2} \\
                   ~ &  \phi, & \phi
                 \end{array}|-1\right).\nonumber
\end{eqnarray}
If we apply \eqref{greene-summation}, then we have
\begin{eqnarray}
&&\sum_{t\in\mathbb{F}_q}\chi_3^2(t)\chi_3\phi(1+t){_{2}F}_{1}\left(\begin{array}{cc}
                       \phi, & \chi_3 \\
                      ~ & \phi
                     \end{array}| t
\right)=q\chi_3\phi(-1){_3F}_{2}\left(\begin{array}{ccc}
                   \phi, & \chi_3, & \chi_3^{2} \\
                   ~ &  \phi, & \phi
                 \end{array}|-1\right).\nonumber
\end{eqnarray}
\begin{remark}
When $d$ is even, we are unable to simplify certain Gauss sums which appear while counting points on the family $Z_{\lambda}$.
It would be interesting to know if similar results like Theorem \ref{MT-4} and Theorem \ref{point-count2} exist when $d$ is even.
\end{remark}

\section{Preliminaries}
\subsection{Gauss sums and Davenport-Hasse relation}
Recall that $\widehat{\mathbb{F}_q^\times}$ denotes the group of all multiplicative characters on $\mathbb{F}_q$.
The \emph{orthogonality relations} for multiplicative characters are listed in the following lemma.
\begin{lemma}\emph{(\cite[Chapter 8]{ireland}).}\label{lemma2} We have
\begin{enumerate}
\item $\displaystyle\sum_{x\in\mathbb{F}_q}\chi(x)=\left\{
                                  \begin{array}{ll}
                                    q-1 & \hbox{if~ $\chi=\varepsilon$;} \\
                                    0 & \hbox{if ~~$\chi\neq\varepsilon$.}
                                  \end{array}
                                \right.$
\item $\displaystyle\sum_{\chi\in \widehat{\mathbb{F}_q^\times}}\chi(x)~~=\left\{
                            \begin{array}{ll}
                              q-1 & \hbox{if~~ $x=1$;} \\
                              0 & \hbox{if ~~$x\neq1$.}
                            \end{array}
                          \right.$
\end{enumerate}
\end{lemma}
\par We now introduce some properties of Gauss sums. For further details, see \cite{evans} noting that we have adjusted
results to take into account $\varepsilon(0)=0$.
Define the additive character $\theta: \mathbb{F}_q \rightarrow \mathbb{C}^{\times}$ by
\begin{align}
\theta(\alpha)=\zeta_p^{\text{tr}(\alpha)}
\end{align}
where $\zeta_p=e^{2\pi i/p}$ and $\text{tr}: \mathbb{F}_q \rightarrow \mathbb{F}_p$ is the trace map given by
$$\text{tr}(\alpha)=\alpha + \alpha^p + \alpha^{p^2}+ \cdots + \alpha^{p^{e-1}}.$$
For $\chi\in \widehat{\mathbb{F}_q^\times}$, the \emph{Gauss sum} is defined by
\begin{align}
g(\chi):=\sum_{x\in \mathbb{F}_q}\chi(x)\zeta_p^{\text{tr}(x)}=\sum_{x\in \mathbb{F}_q}\chi(x)\theta(x).
\end{align}
It is easy to see that
 $\theta(a+b)=\theta(a)\theta(b)$
and
\begin{align}\label{new-eq-2}
 \sum_{x\in \mathbb{F}_q}\theta(x)=0.
\end{align}
Using \eqref{new-eq-2} one easily finds that $g(\varepsilon)=-1$.
\par
The following lemma provides a formula for the multiplicative inverse of a Gauss sum. Let $T$ be a generator of the cyclic group $\widehat{\mathbb{F}_q^\times}$.
\begin{lemma}\emph{(\cite[Eqn. 1.12]{greene}).}\label{fusi3}
If $k\in\mathbb{Z}$ and $T^k\neq\varepsilon$, then
$$g(T^k)g(T^{-k})=q\cdot T^k(-1).$$
\end{lemma}
Using orthogonality, we can write $\theta$ in terms of Gauss sums as given in the following lemma.
\begin{lemma}\emph{(\cite[Lemma 2.2]{Fuselier}).}\label{lemma1}
For all $\alpha \in \mathbb{F}_q^{\times}$, $$\theta(\alpha)=\frac{1}{q-1}\sum_{m=0}^{q-2}g(T^{-m})T^m(\alpha).$$
\end{lemma}
For $\chi,\psi\in\widehat{\mathbb{F}_q^{\times}}$ we define the Jacobi sum by $J(\chi,\psi):=\sum_{t\in\mathbb{F}_q}\chi(t)
\psi(1-t)$. We will use the following relationship between Gauss and Jacobi sums (for example, see \cite[Eqn 1.14]{greene}).
For $\chi,\psi\in\widehat{\mathbb{F}_q^{\times}}$ not both trivial, we have
\begin{align}\label{lemma6}
J(\chi,\psi)=\left\{
               \begin{array}{ll}
                 \frac{g(\chi)g(\psi)}{g(\chi\psi)}, & \hbox{if $\chi\psi\neq\varepsilon$;} \\
                 -\frac{g(\chi)g(\psi)}{q}, & \hbox{if $\chi\psi=\varepsilon$.}
               \end{array}
             \right.
\end{align}
\begin{lemma}\emph{(\cite[Eqn. 1.14]{greene}).}\label{lemma15}
If $T^{m-n}\neq\varepsilon$, then
$$g(T^m)g(T^{-n})=q\left(\begin{array}{c}
                     T^m \\
                     T^n
                   \end{array}\right)g(T^{m-n})T^n(-1)=J(T^m,T^n)g(T^{m-n}).
$$
\end{lemma}
\begin{theorem}\label{thm3}\emph{(\cite[Davenport-Hasse relation]{Lang})}.
Let $p$ be an odd prime and $q=p^e$ for some $e>0$, and let $m$ be a positive integer such that $q\equiv1\pmod{m}$.
For multiplicative characters $\chi,\psi\in\widehat{\mathbb{F}_q^{\times}}$, we have
\begin{align}
\prod_{\chi^m=\varepsilon}g(\chi\psi)=-g(\psi^m)\psi(m^{-m})\prod_{\chi^m=\varepsilon}g(\chi).\notag
\end{align}
\end{theorem}
\subsection{$p$-adic Gamma function, Gross-Koblitz formula and McCarthy's $p$-adic hypergeometric series}
Let $\mathbb{Z}_p$ denote the ring of $p$-adic integers, $\mathbb{Q}_p$ the field of $p$-adic numbers, $\overline{\mathbb{Q}_p}$
the algebraic closure of $\mathbb{Q}_p$, and $\mathbb{C}_p$ the completion of $\overline{\mathbb{Q}_p}$.
It is known that $\mathbb{Z}_p^{\times}$ contains all the $(p-1)$-th roots of unity.
Therefore, we can consider multiplicative characters on $\mathbb{F}_p^\times$
to be maps $\chi: \mathbb{F}_p^{\times} \rightarrow \mathbb{Z}_p^{\times}$.
Let $\omega: \mathbb{F}_p^\times \rightarrow \mathbb{Z}_p^{\times}$ be the Teichm\"{u}ller character.
For $a\in\mathbb{F}_p^\times$, the value $\omega(a)$ is just the $(p-1)$-th root of unity in $\mathbb{Z}_p$ such that $\omega(a)\equiv a \pmod{p}$.
Also, $\widehat{\mathbb{F}_p^{\times}}=\{\omega^j: 0\leq j\leq p-2\}$. Thus, in the $p$-adic setting the Gauss sum $g(\chi)$ takes value in
$\mathbb{Q}_p(\zeta_p)$ for any $\chi \in \widehat{\mathbb{F}_p^{\times}}$.
\par We now recall the definition of $p$-adic gamma function. For further details, see \cite{kob}.
The $p$-adic gamma function $\Gamma_p$ is defined by setting $\Gamma_p(0)=1$, and for positive integer $n$  by
\begin{align}
\Gamma_p(n):=(-1)^n\prod_{\substack{0<j<n\\p\nmid j}}j.\notag
\end{align}
If $x$ and $y$ are two positive integers satisfying $x\equiv y \pmod{p^k\mathbb{Z}}$, then $\Gamma_p(x)\equiv \Gamma_p(y) \pmod{p^k\mathbb{Z}}$.
Therefore, the function
has a unique extension to a continuous function $\Gamma_p: \mathbb{Z}_p \rightarrow \mathbb{Z}_p^{\times}$. If $x\in \mathbb{Z}_p$ and $x\neq 0$, then
$\Gamma_p(x)$ is defined as
\begin{align}
\Gamma_p(x):=\lim_{x_n\rightarrow x}\Gamma_p(x_n),\notag
\end{align}
where $x_n$ runs through any sequence of positive integers $p$-adically approaching $x$. We now introduce Gross-Koblitz formula,
which allows us to relate Gauss sum and the $p$-adic Gamma function. Let $\pi \in \mathbb{C}_p$ be the fixed root of $x^{p-1} + p=0$ which satisfies
$\pi \equiv \zeta_p-1 \pmod{(\zeta_p-1)^2}$.
For $x \in \mathbb{Q}$ we let $\lfloor x\rfloor$ denote the greatest integer less than
or equal to $x$ and $\langle x\rangle$ denote the fractional part of $x$. We have $\langle x\rangle=x-\lfloor x\rfloor$ and $0\leq \langle x\rangle <1$.
Recall that $\overline{\omega}$ denotes the
character inverse of the Teichm\"{u}ller character $\omega$.
\begin{theorem}\emph{(\cite[Gross-Koblitz]{gross}).} For $a\in \mathbb{Z}$, we have
\begin{align}
g(\overline{\omega}^a)=-\pi^{(p-1)\langle\frac{a}{p-1}\rangle}\Gamma_p\left(\left\langle\frac{a}{p-1}\right\rangle\right).\notag
\end{align}
\end{theorem}
We also need the following lemma to prove the main results.
\begin{lemma}\label{lemma4}\emph{(\cite[Eqn 3.4, Lemma 3.4]{BS5}).}
For odd prime $p$ and $0<l\leq p-2$, we have
\begin{eqnarray}
\Gamma_p\left(\frac{l}{p-1}\right)\Gamma_p\left(\left\langle 1-\frac{l}{p-1}\right \rangle\right)=-\overline{\omega}^l(-1).\nonumber
\end{eqnarray}
\end{lemma}
We now state a product formula for the $p$-adic Gamma function.
\begin{lemma}\label{lemma8}\cite[Lemma 4.1]{mccarthy2}.
Let $p$ be an odd prime. For $0\leq l\leq p-2$ and $t\in \mathbb{Z^+}$ with $p\nmid t$, we have
\begin{eqnarray}
\omega(t^{tl})\Gamma_p\left(\left\langle \frac{tl}{p-1}\right\rangle\right)
\prod_{h=1}^{t-1}\Gamma_p\left(\left\langle\frac{h}{t}\right\rangle\right)
&=&\prod_{h=0}^{t-1}\Gamma_p\left(\left\langle\frac{h}{t}+\frac{l}{p-1}\right\rangle\right),\nonumber\\
\omega(t^{-tl})\Gamma_p\left(\left\langle\frac{-tl}{p-1}\right\rangle\right)
\prod_{h=1}^{t-1}\Gamma_p\left(\left\langle \frac{h}{t}\right\rangle\right)
&=&\prod_{h=0}^{t-1}\Gamma_p\left(\left\langle\frac{(1+h)}{t}-\frac{l}{p-1}\right\rangle\right).\nonumber
\end{eqnarray}
\end{lemma}
\par
In \cite{mccarthy3, mccarthy2}, D. McCarthy introduces the notion of hypergeometric series in the $p$-adic setting
which are now famously known as $p$-adic hypergeometric series. The McCarthy's $p$-adic hypergeometric series $_{n}G_{n}[\cdots]$
is defined as follows.
\begin{definition}\cite[Definition 1.1]{mccarthy2} \label{defin1}
Let $p$ be an odd prime and let $t \in \mathbb{F}_p$.
For positive integer $n$ and $1\leq i\leq n$, let $a_i$, $b_i$ $\in \mathbb{Q}\cap \mathbb{Z}_p$.
Then the function $_{n}G_{n}[\cdots]$ is defined by
\begin{align}
&_nG_n\left[\begin{array}{cccc}
             a_1, & a_2, & \ldots, & a_n \\
             b_1, & b_2, & \ldots, & b_n
           \end{array}|t
 \right]:=\frac{-1}{p-1}\sum_{j=0}^{p-2}(-1)^{jn}~~\overline{\omega}^j(t)\notag\\
&\times \prod\limits_{i=1}^n(-p)^{-\lfloor \langle a_i \rangle-\frac{j}{p-1} \rfloor -\lfloor\langle -b_i \rangle +\frac{j}{p-1}\rfloor}
 \frac{\Gamma_p(\langle (a_i-\frac{j}{p-1})\rangle)}{\Gamma_p(\langle a_i \rangle)}
 \frac{\Gamma_p(\langle (-b_i+\frac{j}{p-1}) \rangle)}{\Gamma_p(\langle -b_i \rangle)}.\notag
\end{align}
\end{definition}
\section{Counting points on $Z_{\lambda}: x_1^d+x_2^d=d\lambda x_1x_2^{d-1}$}
In this section, we prove Theorem \ref{point-count} which expresses the number of points over a finite field $\mathbb{F}_p$ on the 0-dimensional variety
$Z_{\lambda}: x_1^d+x_2^d=d\lambda x_1x_2^{d-1}$ in terms of $p$-adic hypergeometric series. We first prove a lemma which will be used to derive
the point count formula.
\begin{lemma}\label{lemma5}
Let $p$ be an odd prime. Then for $0<l\leq p-2$ we have
\begin{align}
\frac{l}{p-1}+\left\langle\frac{(d-1)l}{p-1}\right\rangle+\left\langle\frac{-dl}{p-1}\right\rangle
=1-\sum_{h=1}^{d-1}\lfloor\frac{h}{d}-\frac{l}{p-1}\rfloor
-\sum_{h=1}^{d-2}\lfloor\frac{h}{d-1}+\frac{l}{p-1}\rfloor.\notag
\end{align}
\end{lemma}
\begin{proof}
We have
\begin{align}\label{eq-6}
&\frac{l}{p-1}+\left\langle\frac{(d-1)l}{p-1}\right\rangle+\left\langle\frac{-dl}{p-1}\right\rangle\notag\\
&=\frac{l}{p-1}+\frac{(d-1)l}{p-1}-\frac{dl}{p-1}-\left\lfloor\frac{(d-1)l}{p-1}\right\rfloor-\left\lfloor\frac{-dl}{p-1}\right\rfloor\notag\\
&=-\left\lfloor\frac{(d-1)l}{p-1}\right\rfloor-\left\lfloor\frac{-dl}{p-1}\right\rfloor.
\end{align}
Now, it is enough to prove that
\begin{eqnarray}
\label{eq-8} \left\lfloor\frac{(d-1)l}{p-1}\right\rfloor &=& \sum_{h=1}^{d-2}\left\lfloor\frac{h}{d-1}+\frac{l}{p-1}\right\rfloor,\\
\label{eq-9} \left\lfloor\frac{-dl}{p-1}\right\rfloor &=& \sum_{h=1}^{d-1}\left\lfloor\frac{h}{d}-\frac{l}{p-1}\right\rfloor-1.
\end{eqnarray}
Since $0<\frac{l}{p-1}<1$, we have $0<\frac{(d-1)l}{p-1}<d-1$. Therefore,
$\lfloor\frac{(d-1)l}{p-1}\rfloor \in \{0,1,2,\ldots, d-2\}$. We now prove the lemma by considering some cases.\\
Case 1: If $\lfloor\frac{(d-1)l}{p-1}\rfloor=0$, then $\frac{(d-1)l}{p-1}\neq0$ by the choice of $l$, which yields
$0<\frac{(d-1)l}{p-1}<1$. So, $0<\frac{l}{p-1}<\frac{1}{d-1}$. Therefore,
$\left\lfloor\frac{h}{d-1}+\frac{l}{p-1}\right\rfloor=0$ for $h=1,2,\ldots, d-2$, which gives
\begin{align}
\sum_{h=1}^{d-2}\left\lfloor\frac{h}{d-1}+\frac{l}{p-1}\right\rfloor=0.\notag
\end{align}
Thus, \eqref{eq-8} is true in this case.\\
Case 2: Let $\left\lfloor\frac{(d-1)l}{p-1}\right\rfloor=s$, where $0<s\leq d-2$. Then we have
\begin{align}
s\leq\frac{(d-1)l}{p-1}<s+1,\notag
\end{align}
and this implies
\begin{align}\label{eq-11}
\frac{s}{d-1}\leq\frac{l}{p-1}<\frac{s+1}{d-1}.
\end{align}
Therefore, \eqref{eq-11} implies that whenever $1\leq h\leq d-s-2$ we have $\left\lfloor\frac{h}{d-1}+\frac{l}{p-1}\right\rfloor=0$, which yields
\begin{align}\label{eq-10}
\sum_{h=1}^{d-s-2}\left\lfloor\frac{h}{d-1}+\frac{l}{p-1}\right\rfloor=0.
\end{align}
Also, \eqref{eq-11} implies that for $d-s-1\leq h\leq d-2$, we have
$$\left\lfloor\frac{h}{d-1}+\frac{l}{p-1}\right\rfloor=1,$$ which yields
\begin{align}\label{eq-12}
\sum_{h=d-s-1}^{d-2}\left\lfloor\frac{h}{d-1}+\frac{l}{p-1}\right\rfloor=s.
\end{align}
Combining \eqref{eq-10} and \eqref{eq-12} we find that \eqref{eq-8} is also true in this case. This completes the proof of \eqref{eq-8}.
\par Now, we are going to prove \eqref{eq-9}
using similar arguments. Since $0<\frac{l}{p-1}<1$ in the given range of $l$, so we have
$-d<\frac{-dl}{p-1}<0$. Hence, $\left\lfloor\frac{-dl}{p-1}\right\rfloor \in \{-d,-d+1,\ldots, -1\}$. \\
Case 1: Let $\left\lfloor\frac{-dl}{p-1}\right\rfloor=-d$, then by the choice of $l$, $\frac{-dl}{p-1}\neq-d$, which yields
$-d<\frac{-dl}{p-1}<-d+1$. Thus, we have
\begin{align}\label{eq-13}
-1<\frac{-l}{p-1}<-1+\frac{1}{d}.
\end{align}
Using \eqref{eq-13} we find that $\left\lfloor\frac{h}{d}-\frac{l}{p-1}\right\rfloor=-1$ for $1\leq h\leq d-1$, and this gives
\begin{align}
\sum_{h=1}^{d-1}\left\lfloor\frac{h}{d}-\frac{l}{p-1}\right\rfloor=-(d-1).\notag
\end{align}
Therefore, \eqref{eq-9} is true in this case.\\
Case 2: Let $\left\lfloor\frac{-dl}{p-1}\right\rfloor=-s$, where $s=1,2,\ldots, d-1$. Then we have $-s\leq\frac{-dl}{p-1}<-s+1$, which implies that
\begin{align}\label{eq-14}
\frac{-s}{d}\leq\frac{-l}{p-1}<-\frac{s}{d}+\frac{1}{d}.
\end{align}
Using \eqref{eq-14} we deduce that $\left\lfloor\frac{h}{d}-\frac{l}{p-1}\right\rfloor=-1$ for $1\leq h\leq s-1$
and $\left\lfloor\frac{h}{d}-\frac{l}{p-1}\right\rfloor=0$ for $s\leq h\leq d-1$. Thus, we have
\begin{align}
\sum_{h=1}^{d-1}\left\lfloor\frac{h}{d}-\frac{l}{p-1}\right\rfloor=-(s-1).\notag
\end{align}
Hence, \eqref{eq-9} is also true in this case. Finally, combining \eqref{eq-8} and \eqref{eq-9} we complete the proof of the lemma.
\end{proof}
We now prove the point count formula for the family $Z_{\lambda}: x_1^d+x_2^d=d\lambda x_1x_2^{d-1}$.
\begin{proof}[Proof of Theorem \ref{point-count}]
Let $P(x_1,x_2)=x_1^d+x_2^d-d\lambda x_1x_2^{d-1}$. Let $\#Z_\lambda(\mathbb{F}_p)=\#\{(x_1, x_2)\in \mathbb{F}_p^2: x_1^d+x_2^d=d\lambda x_1x_2^{d-1}\}$
be the number of $\mathbb{F}_p$-points on $Z_{\lambda}$.
If $N_{\mathbb{F}_p}(Z_{\lambda})$ denotes the number of points on $Z_{\lambda}$ in
$\mathbb{P}_{\mathbb{F}_p}^{1}$ then
\begin{eqnarray}\label{new-eq-36}
N_{\mathbb{F}_p}(Z_{\lambda})=\frac{\#Z_\lambda(\mathbb{F}_p)-1}{p-1}.
\end{eqnarray}
Using the identity
\begin{eqnarray}
\sum_{z\in\mathbb{F}_p}\theta(z P(x_1,x_2))=\left\{
                                              \begin{array}{ll}
                                                p, & \hbox{if $P(x_1,x_2)=0$;} \\
                                                0, & \hbox{otherwise,}
                                              \end{array}
                                            \right.
\end{eqnarray}
we have
\begin{eqnarray}\label{eq-1}
p\cdot \#Z_\lambda(\mathbb{F}_p)&=&\sum_{z,x_1,x_2\in \mathbb{F}_p} \theta(z P(x_1,x_2))\nonumber\\
&=&p^2+\sum_{z\in\mathbb{F}_p^{\times}} \theta(0)+\sum_{z,x_1\in\mathbb{F}_p^{\times}}\theta(zx_1^d)
+\sum_{z,x_2\in\mathbb{F}_p^{\times}}\theta(zx_2^d)\nonumber\\
&+&\sum_{z,x_1,x_2\in\mathbb{F}_p^{\times}}\theta(zx_1^d)\theta(zx_2^d)
\theta(-d\lambda zx_1x_2^{d-1})\nonumber\\
&=&p^2+p-1+2\sum_{z,x_1\in\mathbb{F}_p^{\times}}\theta(zx_1^d)+\sum_{z,x_1,x_2\in\mathbb{F}_p^{\times}}\theta(zx_1^d)\theta(zx_2^d)
\theta(-d\lambda zx_1x_2^{d-1})\nonumber\\
&=&p^2+p-1+B+A,
\end{eqnarray}
where $B=2\sum_{z,x_1\in\mathbb{F}_p^{\times}}\theta(zx_1^d)$ and $A=\sum_{z,x_1,x_2\in\mathbb{F}_p^{\times}}\theta(zx_1^d)\theta(zx_2^d)
\theta(-d\lambda zx_1x_2^{d-1}).$ Using Lemma \ref{lemma1} and Lemma \ref{lemma2} we obtain $B=-2(p-1)$.
\par
Again, using Lemma \ref{lemma1} we obtain
\begin{eqnarray}
A&=&\sum_{z,x_1,x_2\in\mathbb{F}_p^{\times}}\theta(zx_1^d)\theta(zx_2^d)\theta(-d\lambda zx_1x_2^{d-1})\nonumber\\
&=&\frac{1}{(p-1)^3}\sum_{z,x_1,x_2\neq0}\sum_{l,m,n=0}^{p-2}g(T^{-l})g(T^{-m})g(T^{-n})T^{l}(zx_1^d)T^{m}(zx_2^d)
T^{n}(-d\lambda zx_1x_2^{d-1})\nonumber\\
&=&\frac{1}{(p-1)^3}\sum_{l,m,n=0}^{p-2}g(T^{-l})g(T^{-m})g(T^{-n})T^{n}(-d\lambda)\sum_{x_1\neq0}T^{dl+n}(x_1)
\nonumber\\
&\times&\sum_{x_2\neq0}T^{dm+(d-1)n}(x_2)\sum_{z\neq0}T^{l+m+n}(z).\nonumber
\end{eqnarray}
From Lemma \ref{lemma2} we observe that the inner sums are non zero only if $n=-dl$ and $m=(d-1)l$. Substituting these values in the above sum we have
\begin{eqnarray}\label{eq-505}
A&=&\sum_{l=0}^{p-2}g(T^{-l})g(T^{-(d-1)l})g(T^{dl})T^{-dl}(-d\lambda).
\end{eqnarray}
Now, taking $T=\omega$ and applying Gross-Koblitz formula we obtain
\begin{eqnarray}
A&=&-\sum_{l=0}^{p-2}\pi^{(p-1)\{\frac{l}{p-1}+\langle\frac{(d-1)l}{p-1}\rangle
+\langle\frac{-dl}{p-1}\rangle\}}\overline{\omega}^{dl}(-d\lambda)\nonumber\\
&\times&\Gamma_p\left(\frac{l}{p-1}\right)
\Gamma_p\left(\left\langle\frac{(d-1)l}{p-1}\right\rangle\right)\Gamma_p\left(\left\langle\frac{-dl}{p-1}\right\rangle\right).\nonumber
\end{eqnarray}
Applying Lemma \ref{lemma8} we deduce that
\begin{eqnarray}\label{eq-2}
\hspace{.5cm}A&=&-\sum_{l=0}^{p-2}\pi^{(p-1)\{\frac{l}{p-1}+\langle\frac{(d-1)l}{p-1}\rangle
+\langle\frac{-dl}{p-1}\rangle\}}~\overline{\omega}^{dl}(-d\lambda)\overline{\omega}^{(d-1)l}(d-1)
\omega^{dl}(d)\nonumber\\
&\times&\Gamma_p\left(\frac{l}{p-1}\right) \frac{\prod_{h=0}^{d-2}\Gamma_p(\langle\frac{h}{d-1}+\frac{l}{p-1}\rangle)}{\prod_{h=1}^{d-2}\Gamma_p(\frac{h}{d-1})}
\frac{\prod_{h=1}^d\Gamma_p(\langle\frac{h}{d}-\frac{l}{p-1}\rangle)}{\prod_{h=1}^{d-1}\Gamma_p(\frac{h}{d})}\nonumber\\
&=&-\sum_{l=0}^{p-2}\pi^{(p-1)\{\frac{l}{p-1}+\langle\frac{(d-1)l}{p-1}\rangle
+\langle\frac{-dl}{p-1}\rangle\}}~\overline{\omega}^{l}((-1)^d\lambda^d(d-1)^{d-1})\Gamma_p\left(\frac{l}{p-1}\right)\nonumber\\
&\times&
\Gamma_p\left(\left\langle1-\frac{l}{p-1}\right\rangle\right)\frac{\prod_{h=0}^{d-2}\Gamma_p(\langle\frac{h}{d-1}+\frac{l}{p-1}\rangle)}{\prod_{h=1}^{d-2}\Gamma_p(\frac{h}{d-1})}
\prod_{h=1}^{d-1}\frac{\Gamma_p(\langle\frac{h}{d}-\frac{l}{p-1}\rangle)}{\Gamma_p(\frac{h}{d})}\nonumber\\
&=&-1-\sum_{l=1}^{p-2}\pi^{(p-1)\{\frac{l}{p-1}+\langle\frac{(d-1)l}{p-1}\rangle
+\langle\frac{-dl}{p-1}\rangle\}}~\overline{\omega}^{l}((-1)^d\lambda^d(d-1)^{d-1})\nonumber\\
&\times&\Gamma_p\left(\frac{l}{p-1}\right)
\Gamma_p\left(\left\langle1-\frac{l}{p-1}\right\rangle\right)\Gamma_p\left(\frac{l}{p-1}\right)\nonumber\\
&\times&\prod_{h=1}^{d-2}\frac{\Gamma_p(\langle\frac{h}{d-1}+\frac{l}{p-1}\rangle)}
{\Gamma_p(\frac{h}{d-1})}\prod_{h=1}^{d-1}\frac{\Gamma_p(\langle\frac{h}{d}-\frac{l}{p-1}\rangle)}{\Gamma_p(\frac{h}{d})}.
\end{eqnarray}
Using Lemma \ref{lemma5} and Lemma \ref{lemma4}, we have
\begin{eqnarray}
A&=&-1+\sum_{l=1}^{p-2}(-p)^{1-\sum_{h=1}^{d-1}\lfloor\frac{h}{d}-\frac{l}{p-1}\rfloor
-\sum_{h=1}^{d-2}\lfloor\frac{h}{d-1}+\frac{l}{p-1}\rfloor}~\overline{\omega}^{l}((-1)^{d-1}\lambda^d(d-1)^{d-1})
\nonumber\\
&\times&\Gamma_p\left(\frac{l}{p-1}\right)\prod_{h=1}^{d-2}\frac{\Gamma_p(\langle\frac{h}{d-1}+\frac{l}{p-1}\rangle)}
{\Gamma_p(\frac{h}{d-1})}\prod_{h=1}^{d-1}\frac{\Gamma_p(\langle\frac{h}{d}-\frac{l}{p-1}\rangle)}
{\Gamma_p(\frac{h}{d})}\nonumber\\
&=&-1-p\sum_{l=1}^{p-2}(-p)^{-\sum_{h=1}^{d-1}\lfloor\frac{h}{d}-\frac{l}{p-1}\rfloor
-\sum_{h=1}^{d-2}\lfloor\frac{h}{d-1}+\frac{l}{p-1}\rfloor}~\overline{\omega}^{l}((-1)^{d-1}\lambda^d(d-1)^{d-1})
\nonumber\\
&\times&\Gamma_p\left(\frac{l}{p-1}\right)\prod_{h=1}^{d-2}\frac{\Gamma_p(\langle\frac{h}{d-1}+\frac{l}{p-1}\rangle)}
{\Gamma_p(\frac{h}{d-1})}\prod_{h=1}^{d-1}\frac{\Gamma_p(\langle\frac{h}{d}-\frac{l}{p-1}\rangle)}
{\Gamma_p(\frac{h}{d})}.\nonumber
\end{eqnarray}
Adding and subtracting the term under summation for $l=0$, we obtain
\begin{eqnarray}
A&&=-1+p-p\sum_{l=0}^{p-2}(-p)^{-\sum_{h=1}^{d-1}\lfloor\frac{h}{d}-\frac{l}{p-1}\rfloor
-\sum_{h=1}^{d-2}\lfloor\frac{h}{d-1}+\frac{l}{p-1}\rfloor}~\overline{\omega}^{l}((-1)^{d-1}\alpha)\nonumber\\
&&\times\Gamma_p\left(\frac{l}{p-1}\right)\prod_{h=1}^{d-2}\frac{\Gamma_p(\langle\frac{h}{d-1}+\frac{l}{p-1}\rangle)}
{\Gamma_p(\frac{h}{d-1})}\prod_{h=1}^{d-1}\frac{\Gamma_p(\langle\frac{h}{d}-\frac{l}{p-1}\rangle)}
{\Gamma_p(\frac{h}{d})}.\nonumber
\end{eqnarray}
Since $\overline{\omega}^{l}(-1)=(-1)^{l}$, we have the following expression for $A$ in terms of the $G$-function.
\begin{eqnarray}\label{eq-3}
A=p-1+p(p-1){_{d-1}G}_{d-1}\left[\begin{array}{cccc}
                                           \frac{1}{d}, & \frac{2}{d}, & \ldots, & \frac{d-1}{d} \\
                                           0, & \frac{1}{d-1}, & \ldots, & \frac{d-2}{d-1}
                                         \end{array}|\alpha
\right],
\end{eqnarray}
where $\alpha=\lambda^d(d-1)^{d-1}$.
Finally, substituting the expressions for $A$ and $B$ in (\ref{eq-1}), and then using \eqref{new-eq-36} we complete the proof.
\end{proof}
\section{Summation identities for the $p$-adic hypergeometric series}
In this section, we prove both the summation identities for the $p$-adic hypergeometric series stated in Theorems \ref{MT1} and \ref{MT2}.
In the following two lemmas, we express certain products
of values of $p$-adic Gamma function in terms of certain character sums.
\begin{lemma}\label{lemma3}
For $1\leq l\leq p-2$ we have
\begin{align}
\frac{(-p)^{-\lfloor\frac{1}{2}+\frac{l}{p-1}\rfloor}
\Gamma_p\left(\langle1-\frac{l}{p-1}\rangle\right)\Gamma_p\left(\langle\frac{1}{2}+\frac{l}{p-1}\rangle\right)}
{\Gamma_p(\frac{1}{2})}=\frac{1}{p}\sum_{t\in \mathbb{F}_p}\overline{\omega}^l(-t)\phi(t(t-1)).\notag
\end{align}
\end{lemma}
\begin{proof}
We have
\begin{eqnarray}
&&\frac{(-p)^{-\lfloor\frac{1}{2}+\frac{l}{p-1}\rfloor}
\Gamma_p\left(\langle1-\frac{l}{p-1}\rangle\right)\Gamma_p\left(\langle\frac{1}{2}+\frac{l}{p-1}\rangle\right)}
{\Gamma_p(\frac{1}{2})}\nonumber\\
&&=\frac{\pi^{-(p-1)\lfloor\frac{1}{2}+\frac{l}{p-1}\rfloor}
\Gamma_p\left(\langle1-\frac{l}{p-1}\rangle\right)\Gamma_p\left(\langle\frac{1}{2}+\frac{l}{p-1}\rangle\right)}
{\Gamma_p(\frac{1}{2})}\nonumber\\
&&=\frac{(\pi)^{-(p-1)(\frac{1}{2}+\frac{l}{p-1})+(p-1)\langle\frac{1}{2}+\frac{l}{p-1}\rangle}
\Gamma_p\left(\langle1-\frac{l}{p-1}\rangle\right)\Gamma_p\left(\langle\frac{1}{2}+\frac{l}{p-1}\rangle\right)}
{\Gamma_p(\frac{1}{2})}\nonumber\\
&&=\frac{\pi^{(p-1)\langle\frac{1}{2}+\frac{l}{p-1}\rangle}\Gamma_p\left(\langle\frac{1}{2}+\frac{l}{p-1}\rangle\right)
\pi^{(p-1)\langle\frac{-l}{p-1}\rangle}\Gamma_p\left(\langle\frac{-l}{p-1}\rangle\right)}
{\pi^{(p-1)(\frac{1}{2})}\Gamma_p(\frac{1}{2})\pi^{(p-1)\{\langle\frac{l}{p-1}\rangle + \langle\frac{-l}{p-1}\rangle\}}}.\nonumber
\end{eqnarray}
Using Gross-Koblitz formula we find that
\begin{eqnarray}\label{neweqn1}
\frac{(-p)^{-\lfloor\frac{1}{2}+\frac{l}{p-1}\rfloor}
\Gamma_p\left(\langle1-\frac{l}{p-1}\rangle\right)\Gamma_p\left(\langle\frac{1}{2}+\frac{l}{p-1}\rangle\right)}
{\Gamma_p(\frac{1}{2})}&=&\frac{-g(\phi\overline{\omega}^l)g(\omega^l)}{\pi^{(p-1)}g(\phi)}.
\end{eqnarray}
Since $1\leq l\leq p-2$, Lemma \ref{fusi3} gives $g(\omega^l)g(\overline{\omega}^l)=p \overline{\omega}^l(-1)$. Then \eqref{neweqn1} reduces to
\begin{eqnarray}\label{eq-15}
&&\frac{(-p)^{-\lfloor\frac{1}{2}+\frac{l}{p-1}\rfloor}
\Gamma_p\left(\langle1-\frac{l}{p-1}\rangle\right)\Gamma_p\left(\langle\frac{1}{2}+\frac{l}{p-1}\rangle\right)}
{\Gamma_p(\frac{1}{2})}\nonumber\\
&&=\frac{-\phi\overline{\omega}^l(-1)g(\phi\overline{\omega}^l)g(\phi)}
{\pi^{p-1}g(\overline{\omega}^l)}\nonumber\\
&&=\frac{1}{p}\frac{\phi\overline{\omega}^l(-1)g(\phi\overline{\omega}^l)g(\phi)}
{g(\overline{\omega}^l)}.
\end{eqnarray}
Now, using \eqref{lemma6} we deduce that
\begin{eqnarray}\label{eq-16}
\frac{\phi\overline{\omega}^l(-1)}{p}\frac{g(\phi\overline{\omega}^l)g(\phi)}
{g(\overline{\omega}^l)}&=&\frac{\phi\overline{\omega}^l(-1)}{p}J(\phi\overline{\omega}^l,\phi)\nonumber\\
&=&\frac{\phi\overline{\omega}^l(-1)}{p}\sum_{t \in \mathbb{F}_p}\phi\overline{\omega}^l(t)\phi(1-t)\nonumber\\
&=&\frac{1}{p}\sum_{t \in \mathbb{F}_p}\phi(t(t-1))\overline{\omega}^l(-t).
\end{eqnarray}
Finally, combining \eqref{eq-15} and \eqref{eq-16} we obtain the desired result.
\end{proof}
\begin{lemma}\label{lemma7}
Let $0\leq l\leq p-2$. Then we have
\begin{align}\label{eq-17}
\frac{(-p)^{-\lfloor\frac{1}{2}-\frac{l}{p-1}\rfloor}
\Gamma_p\left(\langle\frac{l}{p-1}\rangle\right)\Gamma_p\left(\langle\frac{1}{2}-\frac{l}{p-1}\rangle\right)}
{\Gamma_p(\frac{1}{2})}=-\sum_{t \in \mathbb{F}_p}\omega^l(-t)\phi(t(t-1)).
\end{align}
\end{lemma}
\begin{proof}
If we put $l=0$ in both the sides of \eqref{eq-17} then we obtain that the left hand side is 1 and the right hand side is equal to
$-\sum_{t\in \mathbb{F}_p}\phi(t(t-1))$. Using \eqref{lemma6} and Lemma \ref{fusi3}, we easily find that
$\sum_{t\in \mathbb{F}_p}\phi(t(t-1))=-1$, and hence the right hand side of \eqref{eq-17} is also 1. Thus, \eqref{eq-17} is true for $l=0$.
For $1\leq l\leq p-2$, the proof proceeds along similar lines to the proof of Lemma \ref{lemma3} so we omit the details for reasons of brevity.
\end{proof}
\begin{proof}[Proof of Theorem \ref{MT1}]
For $x\in\mathbb{F}_p^{\times}$, we consider the sum
\begin{eqnarray}\label{eq-603}
A_x&=&-1-\sum_{l=1}^{p-2}\pi^{(p-1)\{\frac{l}{p-1}+\langle\frac{(d-1)l}{p-1}\rangle
+\langle\frac{-dl}{p-1}\rangle\}}~\overline{\omega}^{l}(-x)\Gamma_p\left(\frac{l}{p-1}\right)
\Gamma_p\left(\langle1-\frac{l}{p-1}\rangle\right)\nonumber\\
&\times&\Gamma_p\left(\frac{l}{p-1}\right)\prod_{h=1}^{d-2}\frac{\Gamma_p(\langle\frac{h}{d-1}+\frac{l}{p-1}\rangle)}
{\Gamma_p(\frac{h}{d-1})}\prod_{h=1}^{d-1}\frac{\Gamma_p(\langle\frac{h}{d}-\frac{l}{p-1}\rangle)}{\Gamma_p(\frac{h}{d})}.
\end{eqnarray}
Since $d$ is odd, the term $A$ given in \eqref{eq-2} is equal to $A_{\alpha}$ with $\alpha=\lambda^d(d-1)^{d-1}$.
Thus, proceeding similarly as shown in the proof of Theorem \ref{point-count} we deduce that
\begin{eqnarray}\label{eq-604}
A_x=p-1+p(p-1){_{d-1}G}_{d-1}\left[\begin{array}{cccc}
                                           \frac{1}{d}, & \frac{2}{d}, & \ldots, & \frac{d-1}{d} \vspace{.1cm}\\
                                           0, & \frac{1}{d-1}, & \ldots, & \frac{d-2}{d-1}
                                         \end{array}|x\right].
\end{eqnarray}
Also,
\begin{eqnarray}
A_x&=&-1-\sum_{l=1}^{p-2}\pi^{(p-1)\{\frac{l}{p-1}+\langle\frac{(d-1)l}{p-1}\rangle
+\langle\frac{-dl}{p-1}\rangle\}}~\overline{\omega}^{l}(-x)\Gamma_p\left(\frac{l}{p-1}\right)
\Gamma_p\left(\langle1-\frac{l}{p-1}\rangle\right)\nonumber\\
&\times&\Gamma_p\left(\frac{l}{p-1}\right)\prod_{h=1}^{d-2}\frac{\Gamma_p(\langle\frac{h}{d-1}+\frac{l}{p-1}\rangle)}
{\Gamma_p(\frac{h}{d-1})}\prod_{h=1}^{d-1}\frac{\Gamma_p(\langle\frac{h}{d}-\frac{l}{p-1}\rangle)}{\Gamma_p(\frac{h}{d})}.\nonumber\\
&=&-1-\sum_{l=1}^{p-2}\pi^{(p-1)\{\frac{l}{p-1}+\langle\frac{(d-1)l}{p-1}\rangle
+\langle\frac{-dl}{p-1}\rangle\}}~\overline{\omega}^{l}(-x)\Gamma_p\left(\frac{l}{p-1}\right)
\Gamma_p\left(\frac{l}{p-1}\right)\nonumber\\
&\times&\frac{\Gamma_p\left(\langle1-\frac{l}{p-1}\rangle\right)\Gamma_p\left(\langle\frac{1}{2}+\frac{l}{p-1}\rangle\right)}
{\Gamma_p(\frac{1}{2})}
\prod_{\substack{h=1\\h\neq\frac{d-1}{2}}}^{d-2}\frac{\Gamma_p(\langle\frac{h}{d-1}+\frac{l}{p-1}\rangle)}
{\Gamma_p(\frac{h}{d-1})}\prod_{h=1}^{d-1}\frac{\Gamma_p(\langle\frac{h}{d}-\frac{l}{p-1}\rangle)}{\Gamma_p(\frac{h}{d})}.\nonumber
\end{eqnarray}
Using  Lemma \ref{lemma5} we have
\begin{eqnarray}
A_x&=&-1-\sum_{l=1}^{p-2}(-p)^{1-\sum_{h=1}^{d-1}\lfloor\frac{h}{d}-\frac{l}{p-1}\rfloor
-\sum_{h=1}^{d-2}\lfloor\frac{h}{d-1}+\frac{l}{p-1}\rfloor}~\overline{\omega}^{l}(-x)\nonumber\\
&\times&\frac{\Gamma_p\left(\langle1-\frac{l}{p-1}\rangle\right)\Gamma_p\left(\langle\frac{1}{2}+\frac{l}{p-1}\rangle\right)}
{\Gamma_p(\frac{1}{2})}
\Gamma_p\left(\frac{l}{p-1}\right)\Gamma_p\left(\frac{l}{p-1}\right)\nonumber\\
&\times&\prod_{\substack{h=1\\h\neq\frac{d-1}{2}}}^{d-2}\frac{\Gamma_p(\langle\frac{h}{d-1}+\frac{l}{p-1}\rangle)}
{\Gamma_p(\frac{h}{d-1})}\prod_{h=1}^{d-1}\frac{\Gamma_p(\langle\frac{h}{d}-\frac{l}{p-1}\rangle)}
{\Gamma_p(\frac{h}{d})}.\nonumber\\
&=&-1-\sum_{l=1}^{p-2}(-p)^{1-\sum_{h=1}^{d-1}\lfloor\frac{h}{d}-\frac{l}{p-1}\rfloor
-\sum_{\substack{h=1\\h\neq\frac{d-1}{2}}}^{d-2}\lfloor\frac{h}{d-1}+\frac{l}{p-1}\rfloor}
~\overline{\omega}^{l}(-x)\nonumber\\
&\times&\frac{(-p)^{-\lfloor\frac{1}{2}+\frac{l}{p-1}\rfloor}
\Gamma_p\left(\langle1-\frac{l}{p-1}\rangle\right)\Gamma_p\left(\langle\frac{1}{2}+\frac{l}{p-1}\rangle\right)}
{\Gamma_p(\frac{1}{2})}\Gamma_p\left(\frac{l}{p-1}\right)\Gamma_p\left(\frac{l}{p-1}\right)\nonumber\\
&\times&\prod_{\substack{h=1\\h\neq\frac{d-1}{2}}}^{d-2}\frac{\Gamma_p(\langle\frac{h}{d-1}+\frac{l}{p-1}\rangle)}
{\Gamma_p(\frac{h}{d-1})}\prod_{h=1}^{d-1}\frac{\Gamma_p(\langle\frac{h}{d}-\frac{l}{p-1}\rangle)}
{\Gamma_p(\frac{h}{d})}.\nonumber
\end{eqnarray}
Lemma \ref{lemma3} yields
\begin{eqnarray}
A_x&=&-1+\sum_{l=1}^{p-2}(-p)^{-\sum_{h=1}^{d-1}\lfloor\frac{h}{d}-\frac{l}{p-1}\rfloor
-\sum_{\substack{h=1\\h\neq\frac{d-1}{2}}}^{d-2}\lfloor\frac{h}{d-1}+\frac{l}{p-1}\rfloor}
~\overline{\omega}^{l}(-x)\nonumber\\
&\times&\sum_{t\in \mathbb{F}_p}\overline{\omega}^l(-t)\phi(t(t-1))\Gamma_p\left(\frac{l}{p-1}\right)\Gamma_p\left(\frac{l}{p-1}\right)\nonumber\\
&\times&\prod_{\substack{h=1\\h\neq\frac{d-1}{2}}}^{d-2}\frac{\Gamma_p(\langle\frac{h}{d-1}+\frac{l}{p-1}\rangle)}
{\Gamma_p(\frac{h}{d-1})}\prod_{h=1}^{d-1}\frac{\Gamma_p(\langle\frac{h}{d}-\frac{l}{p-1}\rangle)}
{\Gamma_p(\frac{h}{d})}.\nonumber
\end{eqnarray}
The term under summation for $l=0$ is $\sum_{t\in \mathbb{F}_p}\phi(t(t-1))$. Using \eqref{lemma6} and Lemma \ref{fusi3}, we easily find that
$\sum_{t\in \mathbb{F}_p}\phi(t(t-1))=-1$. Thus,
\begin{eqnarray}
&&\hspace{-.6cm}A_x=\sum_{t \in \mathbb{F}_p}\phi(t(t-1))\sum_{l=0}^{p-2}(-p)^{-\sum_{h=1}^{d-1}\lfloor\frac{h}{d}-\frac{l}{p-1}\rfloor
-\sum_{\substack{h=1\\h\neq\frac{d-1}{2}}}^{d-2}\lfloor\frac{h}{d-1}+\frac{l}{p-1}\rfloor}
~\overline{\omega}^{l}(xt)\nonumber\\
&&\times\Gamma_p\left(\frac{l}{p-1}\right)\Gamma_p\left(\frac{l}{p-1}\right)
\prod_{\substack{h=1\\h\neq\frac{d-1}{2}}}^{d-2}\frac{\Gamma_p(\langle\frac{h}{d-1}+\frac{l}{p-1}\rangle)}
{\Gamma_p(\frac{h}{d-1})}\prod_{h=1}^{d-1}\frac{\Gamma_p(\langle\frac{h}{d}-\frac{l}{p-1}\rangle)}
{\Gamma_p(\frac{h}{d})}\nonumber\\
&&=-(p-1)\sum_{t \in \mathbb{F}_p}\phi(t(t-1))\times \nonumber\\
&&{_{d-1}G}_{d-1}\left[\begin{array}{cccccccccc}
                       \frac{1}{d}, \hspace{-.2cm} & \frac{2}{d}, \hspace{-.2cm} & \ldots, \hspace{-.2cm} & \frac{\frac{d-1}{2}-1}{d},
                       \hspace{-.2cm}& \frac{\frac{d-1}{2}}{d}, \hspace{-.2cm}
 & \frac{\frac{d-1}{2}+1}{d}, \hspace{-.2cm} & \ldots, \hspace{-.2cm} & \frac{d-3}{d}, \hspace{-.2cm} & \frac{d-2}{d}, \hspace{-.2cm} & \frac{d-1}{d}
 \vspace{.1cm}\\
 \frac{1}{d-1}, \hspace{-.2cm} & \frac{2}{d-1}, & \hspace{-.2cm}\ldots, \hspace{-.2cm} & \frac{\frac{d-1}{2}-1}{d-1},
 \hspace{-.2cm} & \frac{\frac{d-1}{2}+1}{d-1}, \hspace{-.2cm} & \frac{\frac{d-1}{2}+2}{d-1}, & \hspace{-.2cm} \ldots, \hspace{-.2cm} & \frac{d-2}{d-1}, \hspace{-.2cm}
 & 0, \hspace{-.2cm} &0
                     \end{array}|xt
\right].\nonumber
\end{eqnarray}
Finally, combining \eqref{eq-604} and the above expression for $A_x$ we derive the required summation identity.
\end{proof}
\begin{proof}[Proof of Theorem \ref{MT2}]
For $x\in\mathbb{F}_p^{\times}$, we consider the sum
\begin{eqnarray}\label{eq-610}
A_x&=&-1-\sum_{l=1}^{p-2}\pi^{(p-1)\{\frac{l}{p-1}+\langle\frac{(d-1)l}{p-1}\rangle
+\langle\frac{-dl}{p-1}\rangle\}}~\overline{\omega}^{l}(x)\Gamma_p\left(\frac{l}{p-1}\right)
\Gamma_p\left(\langle1-\frac{l}{p-1}\rangle\right)\nonumber\\
&\times&\Gamma_p\left(\frac{l}{p-1}\right)\prod_{h=1}^{d-2}\frac{\Gamma_p(\langle\frac{h}{d-1}+\frac{l}{p-1}\rangle)}
{\Gamma_p(\frac{h}{d-1})}\prod_{h=1}^{d-1}\frac{\Gamma_p(\langle\frac{h}{d}-\frac{l}{p-1}\rangle)}{\Gamma_p(\frac{h}{d})}.
\end{eqnarray}
Since $d$ is even, the term $A$ given in \eqref{eq-2} is equal to $A_{\alpha}$ with $\alpha=\lambda^d(d-1)^{d-1}$.
Thus, proceeding similarly as shown in the proof of Theorem \ref{point-count} we deduce that
\begin{eqnarray}\label{eq-611}
A_x=p-1+p(p-1){_{d-1}G}_{d-1}\left[\begin{array}{cccc}
                                           \frac{1}{d}, & \frac{2}{d}, & \ldots, & \frac{d-1}{d} \\
                                           0, & \frac{1}{d-1}, & \ldots, & \frac{d-2}{d-1}
                                         \end{array}|x\right].
\end{eqnarray}
Also,
\begin{eqnarray}
A_x&=&-1-\sum_{l=1}^{p-2}\pi^{(p-1)\{\frac{l}{p-1}+\langle\frac{(d-1)l}{p-1}\rangle
+\langle\frac{-dl}{p-1}\rangle\}}~\overline{\omega}^{l}(x)\Gamma_p\left(\frac{l}{p-1}\right)
\Gamma_p\left(\langle1-\frac{l}{p-1}\rangle\right)\nonumber\\
&\times&\Gamma_p\left(\frac{l}{p-1}\right)\prod_{h=1}^{d-2}\frac{\Gamma_p(\langle\frac{h}{d-1}+\frac{l}{p-1}\rangle)}
{\Gamma_p(\frac{h}{d-1})}\prod_{h=1}^{d-1}\frac{\Gamma_p(\langle\frac{h}{d}-\frac{l}{p-1}\rangle)}
{\Gamma_p(\frac{h}{d})}\nonumber\\
&=&-1-\sum_{l=1}^{p-2}\pi^{(p-1)\{\frac{l}{p-1}+\langle\frac{(d-1)l}{p-1}\rangle
+\langle\frac{-dl}{p-1}\rangle\}}~\overline{\omega}^{l}(x)\Gamma_p\left(\frac{l}{p-1}\right)
\Gamma_p\left(\langle1-\frac{l}{p-1}\rangle\right)\nonumber\\
&\times&\frac{\Gamma_p\left(\frac{l}{p-1}\right)\Gamma_p(\langle\frac{1}{2}-\frac{l}{p-1}\rangle)}{\Gamma_p(\frac{1}{2})}
\prod_{h=1}^{d-2}\frac{\Gamma_p(\langle\frac{h}{d-1}+\frac{l}{p-1}\rangle)}
{\Gamma_p(\frac{h}{d-1})}\prod_{\substack{h=1\\h\neq\frac{d}{2}}}^{d-1}\frac{\Gamma_p(\langle\frac{h}{d}-\frac{l}{p-1}\rangle)}
{\Gamma_p(\frac{h}{d})}.\nonumber
\end{eqnarray}
We now apply Lemma \ref{lemma5} and Lemma \ref{lemma4} to obtain
\begin{eqnarray}
A_x&=&-1+\sum_{l=1}^{p-2}(-p)^{1-\sum_{h=1}^{d-1}\lfloor\frac{h}{d}-\frac{l}{p-1}\rfloor
-\sum_{h=1}^{d-2}\lfloor\frac{h}{d-1}+\frac{l}{p-1}\rfloor}
~\overline{\omega}^{l}(-x)\nonumber\\
&\times&
\frac{\Gamma_p\left(\frac{l}{p-1}\right)\Gamma_p(\langle\frac{1}{2}-\frac{l}{p-1}\rangle)}{\Gamma_p(\frac{1}{2})}
\prod_{h=1}^{d-2}\frac{\Gamma_p(\langle\frac{h}{d-1}+\frac{l}{p-1}\rangle)}
{\Gamma_p(\frac{h}{d-1})}\prod_{\substack{h=1\\h\neq\frac{d}{2}}}^{d-1}\frac{\Gamma_p(\langle\frac{h}{d}-\frac{l}{p-1}\rangle)}
{\Gamma_p(\frac{h}{d})}\nonumber\\
&=&-1+\sum_{l=1}^{p-2}(-p)^{1-\sum_{h=1,~h\neq\frac{d}{2}}^{d-1}\lfloor\frac{h}{d}-\frac{l}{p-1}\rfloor
-\sum_{h=1}^{d-2}\lfloor\frac{h}{d-1}+\frac{l}{p-1}\rfloor}
~\overline{\omega}^{l}(-x)\nonumber\\
&\times&
\frac{(-p)^{-\lfloor\frac{1}{2}-\frac{l}{p-1}\rfloor}\Gamma_p\left(\frac{l}{p-1}\right)
\Gamma_p(\langle\frac{1}{2}-\frac{l}{p-1}\rangle)}{\Gamma_p(\frac{1}{2})}\nonumber\\
&\times&\prod_{h=1}^{d-2}\frac{\Gamma_p(\langle\frac{h}{d-1}+\frac{l}{p-1}\rangle)}
{\Gamma_p(\frac{h}{d-1})}\prod_{\substack{h=1\\h\neq\frac{d}{2}}}^{d-1}\frac{\Gamma_p(\langle\frac{h}{d}-\frac{l}{p-1}\rangle)}
{\Gamma_p(\frac{h}{d})}.\nonumber
\end{eqnarray}
Adding and subtracting the term under summation for $l=0$, and then applying Lemma \ref{lemma7} we deduce that
\begin{eqnarray}
A_x&=&-1+p+p\sum_{t\in \mathbb{F}_p}\phi(t(t-1))\sum_{l=0}^{p-2}(-p)^{-\sum_{h=1,~h\neq\frac{d}{2}}^{d-1}\lfloor\frac{h}{d}-\frac{l}{p-1}
\rfloor-\sum_{h=1}^{d-2}\lfloor\frac{h}{d-1}+\frac{l}{p-1}\rfloor}\nonumber\\
&\times&\overline{\omega}^{l}(-x)\omega^{l}(- t)
\prod_{h=1}^{d-2}\frac{\Gamma_p(\langle\frac{h}{d-1}+\frac{l}{p-1}\rangle)}
{\Gamma_p(\frac{h}{d-1})}\prod_{\substack{h=1\\h\neq\frac{d}{2}}}^{d-1}\frac{\Gamma_p(\langle\frac{h}{d}-\frac{l}{p-1}\rangle)}
{\Gamma_p(\frac{h}{d})}\nonumber\\
&=&-1+p-p(p-1)\sum_{t \in \mathbb{F}_p^{\times}}\phi(t(t-1))\nonumber\\
&\times& \hspace{-.2cm}{_{d-2}G}_{d-2}\left[\begin{array}{ccccccccc}
\frac{1}{d},\hspace{-.2cm} & \frac{2}{d}, & \hspace{-.2cm} \ldots, & \frac{\frac{d}{2}-1}{d}, & \frac{\frac{d}{2}+1}{d}, & \frac{\frac{d}{2}+2}{d},
& \hspace{-.2cm} \ldots, & \frac{d-2}{d}, & \hspace{-.1cm} \frac{d-1}{d}\vspace{.1cm} \\
\frac{1}{d-1},\hspace{-.2cm} & \frac{2}{d-1}, & \hspace{-.2cm} \ldots, &\frac{\frac{d}{2}-1}{d-1}, & \frac{\frac{d}{2}}{d-1},
&\frac{\frac{d}{2}+1}{d-1}, & \hspace{-.2cm} \ldots, & \frac{d-3}{d-1}, & \hspace{-.1cm} \frac{d-2}{d-1}
\end{array}|\frac{x}{t}\right].\nonumber
\end{eqnarray}
Finally, combining \eqref{eq-611} and the above expression, and then replacing $1/t$ by $t$ we complete the proof.
\end{proof}
\section{Transformations and special values of the $p$-adic hypergeometric series}
In this section, we derive transformations for the $p$-adic hypergeometric series. We use these transformations to find
certain special values of the $G$-function. In \cite{BS-FFA}, we express the number of distinct zeros of the polynomials $x^d+ax+b$ and $x^d+ax^{d-1}+b$
over a finite field in terms of McCarthy's $p$-adic hypergeometric series. We use certain Gauss sums evaluations from \cite{BS-FFA} in the proof of
Theorem \ref{MT-6} below.
\begin{proof}[Proof of Theorem \ref{MT-6}]
For $\lambda\in\mathbb{F}_p^{\times}$, we consider the sum
\begin{eqnarray}\label{eq-503}
A_{\lambda}=\sum_{l=0}^{p-2}g(T^{-l})g(T^{-(d-1)l})g(T^{dl})T^l\left({\frac{(-1)^d(d-1)^{d-1}}{d^d\lambda}}\right).
\end{eqnarray}
Since \eqref{eq-505} and \eqref{eq-503} contain the same Gauss sums, so proceeding similarly as shown in the proof of Theorem \ref{point-count},
we deduce that
\begin{eqnarray}\label{eq-504}
A_{\lambda}=p-1+p(p-1)~{_{d-1}G}_{d-1}\left[\begin{array}{cccc}
                                           \frac{1}{d}, & \frac{2}{d}, & \ldots, & \frac{d-1}{d} \vspace{.1cm}\\
                                           0, & \frac{1}{d-1}, & \ldots, & \frac{d-2}{d-1}
                                         \end{array}|\lambda\right].
\end{eqnarray}
Now, if $d$ is even, then replacing $l$ by $l-\frac{p-1}{2}$ in \eqref{eq-503} we have
\begin{eqnarray}\label{eq-506}
&&A_{\lambda}=\phi(\lambda(d-1))\sum_{l=0}^{p-2}g(T^{-l+\frac{p-1}{2}})g(T^{-(d-1)l+\frac{p-1}{2}})g(T^{dl}) T^l\left(\frac{(d-1)^{d-1}}{d^d\lambda}\right).
\end{eqnarray}
We observe that the Gauss sums present in \eqref{eq-506} are the same Gauss sums appeared in \cite[Eqn 11]{BS-FFA}.
Therefore, proceeding similarly as shown in the proof of \cite[Theorem 1.2]{BS-FFA}, we deduce that
\begin{eqnarray}\label{eq-507}
A_{\lambda}&=&p-1+p(p-1)\phi(-\lambda(d-1))\nonumber\\
&&\times{_{d-1}}G_{d-1}\left[\begin{array}{ccccccc}
                       \frac{1}{2(d-1)},\hspace{-.2cm} & \frac{3}{2(d-1)}, \hspace{-.2cm}& \ldots, \hspace{-.2cm}& \frac{d-1}{2(d-1)},
                       \hspace{-.2cm}& \frac{d+1}{2(d-1)}, \hspace{-.2cm}& \ldots, \hspace{-.2cm} & \frac{2(d-1)-1}{2(d-1)}\vspace{.1cm}\\
                       0, \hspace{-.2cm}& \frac{1}{d}, \hspace{-.2cm}& \ldots, \hspace{-.2cm}& \frac{\frac{d}{2}-1}{d}, \hspace{-.2cm}
                       & \frac{\frac{d}{2}+1}{d}, \hspace{-.2cm}& \ldots, \hspace{-.2cm} & \frac{d-1}{d} \end{array}|\frac{1}{\lambda}
\right].
\end{eqnarray}
Combining \eqref{eq-504} and \eqref{eq-507} we obtain the desired transformation when $d$ is even.
\par If $d$ is odd, then replacing $l$ by $l-\frac{p-1}{2}$ in \eqref{eq-503} we have
\begin{eqnarray}\label{eq-508}
&&A_{\lambda}=\phi(-d\lambda)\sum_{l=0}^{p-2}g(T^{-l+\frac{p-1}{2}})g(T^{-(d-1)l})g(T^{dl+\frac{p-1}{2}})
T^l\left(\frac{-(d-1)^{d-1}}{d^d\lambda}\right).
\end{eqnarray}
Again, we observe that the Gauss sums present in \eqref{eq-508} are the same Gauss sums appeared in \cite[Eqn 22]{BS-FFA}.
Therefore, proceeding similarly as shown in the proof of \cite[Theorem 1.3]{BS-FFA}, we deduce that
\begin{eqnarray}\label{eq-509}
A_{\lambda}&=&p-1+p(p-1)\phi(d\lambda)\nonumber\\
&\times&{_{d-1}}G_{d-1}\left[\begin{array}{cccccccc}
                      0, \hspace{-.2cm}& \frac{1}{d-1}, \hspace{-.2cm}& \ldots, \hspace{-.2cm}& \frac{d-3}{2(d-1)}, \hspace{-.2cm}& \frac{d-1}{2(d-1)},
                      \hspace{-.2cm}&\ldots, \hspace{-.2cm}& \frac{d-3}{d-1}, \hspace{-.2cm}& \frac{d-2}{d-1} \vspace{.1cm}\\
                      \frac{1}{2d}, \hspace{-.2cm}& \frac{3}{2d}, \hspace{-.2cm}& \ldots, \hspace{-.2cm}& \frac{d-2}{2d}, \hspace{-.2cm}&
                      \frac{d+2}{2d}, \hspace{-.2cm}& \ldots, \hspace{-.2cm}& \frac{2d-3}{2d}, \hspace{-.2cm}& \frac{2d-1}{2d}
                    \end{array}
|\frac{1}{\lambda}
\right].
\end{eqnarray}
Finally, combining \eqref{eq-504} and \eqref{eq-509} we obtain the desired transformation when $d$ is odd. This completes the proof of the theorem.
\end{proof}
\begin{remark}
For $\lambda \neq 0$, the number of points in $\mathbb{P}^1_{\mathbb{F}_q}$ over a finite field $\mathbb{F}_q$ on the family
$Z_{\lambda}: x_1^d+x_2^d=d\lambda x_1x_2^{d-1}$ is equal to the number of distinct zeros of the polynomial $x^d-d\lambda x +1$ over $\mathbb{F}_q$.
Therefore, using \cite[Theorem 1.2 and Theorem 1.3]{BS-FFA} and Theorem \ref{point-count} we obtain the transformations stated in Theorem \ref{MT-6} for
certain values of $\lambda$, namely $\lambda^d(d-1)^{d-1}$.
\end{remark}
\begin{proof}[Proof of Theorem \ref{MT-5}]
Putting $d=3$ in Theorem \ref{MT-6}, we find that
\begin{eqnarray}\label{eq-500}
{_2G}_2\left[\begin{array}{cc}
               \frac{1}{3}, & \frac{2}{3} \vspace{.1 cm} \\
               0, & \frac{1}{2}
             \end{array}|\frac{4}{27}\right]={_2G}_2\left[\begin{array}{cc}
               0, & \frac{1}{2} \vspace{.1 cm} \\
               \frac{1}{6}, & \frac{5}{6}
             \end{array}|\frac{27}{4}\right]
\end{eqnarray}
for $p>3$. Now, from \cite[Theorem 4.6]{BS5} we have
\begin{eqnarray}\label{eq-501}
&&{_4G}_4\left[\begin{array}{cccc}
         0, & \frac{1}{4}, & \frac{1}{2}, & \frac{3}{4} \vspace{.1 cm} \\
         \frac{1}{10}, & \frac{3}{10}, & \frac{7}{10}, & \frac{9}{10}
       \end{array}|\frac{-5^5}{4^4}\right]\nonumber\\
       &&=\phi(-1)+\phi(3)+\phi(-1){_2G}_2\left[\begin{array}{cc}
               0, & \frac{1}{2} \vspace{.1 cm} \\
               \frac{1}{6}, & \frac{5}{6}
             \end{array}|\frac{27}{4}\right]
\end{eqnarray} for $p>7$ and $p\neq23$.
Combining \eqref{eq-500} and \eqref{eq-501} we readily obtain the first identity.
Again, if we apply Theorem \ref{MT-6} for $d=5$, then for $p=3$ and $p>5$ we have
\begin{eqnarray}\label{eq-502}
&&{_4G}_4\left[\begin{array}{cccc}
         \frac{1}{5}, & \frac{2}{5}, & \frac{3}{5}, & \frac{4}{5} \vspace{.1 cm} \\
         0, & \frac{1}{4}, & \frac{1}{2}, & \frac{3}{4}
       \end{array}|-\frac{4^4}{5^5}\right]\nonumber\\
       &&=\phi(-1){_4G}_4\left[\begin{array}{cccc}
         0, & \frac{1}{4}, & \frac{1}{2}, & \frac{3}{4} \vspace{.1 cm} \\
         \frac{1}{10}, & \frac{3}{10}, & \frac{7}{10}, & \frac{9}{10}
       \end{array}|-\frac{5^5}{4^4}\right].
\end{eqnarray}
Combining \eqref{eq-501}, \eqref{eq-502} and \eqref{eq-500} we obtain the second set of transformations.
\end{proof}
\par
In \cite{BS5}, the authors with D. McCarthy find certain special values of the $G$-function. We use the transformations given in Theorem
\ref{MT-6} to find some new values of the $G$-function.
\begin{theorem}\label{sv1}
Let $a,b,c\in\mathbb{F}_p^{\times}$ be such that $a+b+c=0$ and $ab+bc+ca\neq 0$. Then, for $p\geq 5$, we have
\begin{eqnarray}\label{sv1-eq1}
 {_{2}G}_{2}\left[\begin{array}{cc}
                         \frac{1}{3}, & \frac{2}{3} \vspace{.1cm}\\
                         0, & \frac{1}{2}
                       \end{array}|-\frac{4(ab+bc+ca)^3}{27a^2b^2c^2}\right]=A,
                       \end{eqnarray}
where $A=2$ if all of $a, b, c$ are distinct and $A=1$ if exactly two of $a, b, c$ are equal.
\par If $a,b,c\in\mathbb{F}_p^{\times}$ are such that $ab+bc+ca=0$ and $a+b+c\neq 0$, then, for $p\geq 5$, we have
\begin{eqnarray}\label{sv1-eq2}
 {_{2}}G_{2}\left[\begin{array}{cc}
                         \frac{1}{3}, & \frac{2}{3} \vspace{.1cm}\\
                         0, & \frac{1}{2}
                       \end{array}|-\frac{4(a+b+c)^3}{27abc}\right]=A.
\end{eqnarray}
\end{theorem}
\begin{proof}
Let $a+b+c=0$ and $ab+bc+ca\neq 0$. Then, from \cite[Theorem 4.1]{BS5}, for $p\geq 5$, we have
\begin{eqnarray}\label{eq-600}
 {_2}G_{2}\left[\begin{array}{cc}
                         0, & \frac{1}{2} \vspace{.1cm}\\
                         \frac{1}{6}, & \frac{5}{6}
                       \end{array}|-\frac{27a^2b^2c^2}{4(ab+bc+ca)^3}\right]=
             A\cdot \phi(-(ab+bc+ca)).
\end{eqnarray}
Now, applying Theorem \ref{MT-6} for $d=3$ and $\lambda=-\frac{4(ab+bc+ca)^3}{27a^2b^2c^2}$, and
then comparing the result with \eqref{eq-600} we derive \eqref{sv1-eq1}.
Again, if $ab+bc+ca=0$ and $a+b+c\neq 0$, then \cite[Theorem 4.1]{BS5} gives
\begin{eqnarray}\label{eq-601}
 {_2}G_{2}\left[\begin{array}{cc}
                         0, & \frac{1}{2} \vspace{.1cm} \\
                         \frac{1}{6}, & \frac{5}{6}
                       \end{array}|-\frac{27abc}{4(a+b+c)^3}\right]=A\cdot \phi(-abc(a+b+c))
\end{eqnarray}
for $p\geq5$.
We now apply Theorem \ref{MT-6} for $d=3$ and $\lambda=-\frac{4(a+b+c)^3}{27abc}$, and
then compare the result with \eqref{eq-601} to derive \eqref{sv1-eq2}. This completes the proof of the theorem.
\end{proof}
\begin{example}
If we put $a=b=1$ and $c=-2$ in \eqref{sv1-eq1}, then for $p\geq5$ we have
\begin{eqnarray}
{_{2}G}_{2}\left[\begin{array}{cc}
                         \frac{1}{3}, & \frac{2}{3} \vspace{.1cm}\\
                         0, & \frac{1}{2}
                       \end{array}|1\right]=1.\nonumber
\end{eqnarray}
If we put $a=1$, $b=2$ and $c=-3$ in \eqref{sv1-eq1}, then for $p\geq5$ we have
\begin{eqnarray}
{_{2}G}_{2}\left[\begin{array}{cc}
                         \frac{1}{3}, & \frac{2}{3} \vspace{.1cm}\\
                         0, & \frac{1}{2}
                       \end{array}|\frac{343}{243}\right]=2.\nonumber
\end{eqnarray}
\end{example}
\begin{theorem}\label{sv-2}
If $p\geq 5$, then we have
\begin{eqnarray}
{_{3}}G_{3}\left[\begin{array}{ccc}
             \frac{1}{4}, & \frac{1}{2}, & \frac{3}{4} \vspace{.1cm}\\
             0, & \frac{1}{3}, & \frac{2}{3}
           \end{array}|1\right]=1+\phi(-2).\nonumber
\end{eqnarray}
\end{theorem}
\begin{proof}
If $p\geq5$, then from \cite[Theorem 4.5]{BS5} we have
\begin{eqnarray}\label{eq-602}
{_{3}}G_{3}\left[\begin{array}{ccc}
             \frac{1}{6}, & \frac{1}{2}, & \frac{5}{6} \vspace{.1cm}\\
             0, & \frac{1}{4}, & \frac{3}{4}
           \end{array}|1\right]=\phi(-3)+\phi(6).
\end{eqnarray}
If we use Theorem \ref{MT-6} for $d=4$ and $\lambda=1$, then we have
\begin{eqnarray}\label{eq-602-new}
{_{3}}G_{3}\left[\begin{array}{ccc}
             \frac{1}{4}, & \frac{1}{2}, & \frac{3}{4} \vspace{.1cm}\\
             0, & \frac{1}{3}, & \frac{2}{3}
           \end{array}|1\right]
=\phi(-3){_{3}}G_{3}\left[\begin{array}{ccc}
             \frac{1}{6}, & \frac{1}{2}, & \frac{5}{6} \vspace{.1cm}\\
             0, & \frac{1}{4}, & \frac{3}{4}
           \end{array}|1\right].
\end{eqnarray}
Now, \eqref{eq-602} and \eqref{eq-602-new} readily gives us the desired special value.
\end{proof}
\section{Summation identities for Greene's hypergeometric series}
In this section we prove the point count formula for the family $Z_{\lambda}$ and the
summation identity for Greene's hypergeometric series.
We first prove two lemmas which will be used to prove our main results. The following lemma is a special case of Davenport-Hasse relation.
\begin{lemma}\label{lemma9}
Let $d$ be a positive integer and let $p$ be an odd prime and $q=p^r$ such that $q\equiv1\pmod{d}$. Then for $t\in\{1,-1\}$ and $l\in\mathbb{Z}$ we have
\begin{align}
\prod_{i=0}^{d-1}g(T^{l+t\frac{i(q-1)}{d}})=\left\{
 \begin{array}{ll}
q^{\frac{d-1}{2}}T^{\frac{(d-1)(d+1)(q-1)}{8d}}(-1)T^{-ld}(d)g(T^{ld}), & \hbox{if $d\geq1$ is odd ;} \\
                                               q^{\frac{d-2}{2}}g(\phi)T^{\frac{(d-2)(q-1)}{8}}(-1)T^{-ld}(d)g(T^{ld}), & \hbox{if $d\geq2$ is even.}
                                             \end{array}
                                           \right.
\notag
\end{align}
\end{lemma}
\begin{proof}
The lemma readily follows by putting $m=d$ in Theorem \ref{thm3}, and then applying Lemma \ref{fusi3}.
\end{proof}
\begin{lemma}\label{lemma10}
Let $0\leq l\leq q-2$. Then we have
\begin{align}\label{eq-18}
\frac{g(T^l)g(T^{-l}\phi)}{g(\phi)}=\sum_{t\in\mathbb{F}_q}\phi(t(t-1))T^{-l}(-t).
\end{align}
\end{lemma}
\begin{proof}
If we put $l=0$ on the left hand side of \eqref{eq-18}, then we have $\frac{g(\varepsilon)g(\phi)}{g(\phi)}=-1$.
Also, if we simplify the expression on the right hand side of \eqref{eq-18} for $l=0$, then we have
\begin{align}
&\sum_{t\in\mathbb{F}_q}\phi(t(t-1))
=\phi(-1)\sum_{t\in\mathbb{F}_q}\phi(t)\phi(1-t)=\phi(-1)J(\phi,\phi).\notag
\end{align}
Using \eqref{lemma6} and then Lemma \ref{fusi3} we obtain that the above sum is equal to $-1$. Thus the right hand side of \eqref{eq-18} is also $-1$.
So, the result is true for $l=0$.
Now, for $l\neq0$ using Lemma \ref{fusi3} and then \eqref{lemma6} we have
\begin{align}
\frac{g(T^l)g(T^{-l}\phi)}{g(\phi)}&=\frac{\phi T^l(-1)g(T^{-l}\phi)g(\phi)}{g(T^{-l})}
=\phi T^l(-1)J(T^{-l}\phi,\phi)\notag\\&=\phi T^l(-1)\sum_{t\in\mathbb{F}_q}T^{-l}\phi(t)\phi(1-t)
=\sum_{t\in\mathbb{F}_q}\phi(t(t-1))T^{-l}(-t).\notag
\end{align}
This completes the proof of the lemma.
\end{proof}
We now prove Theorem \ref{point-count2} which will be used to deduce the summation identity.
\begin{proof}[Proof of Theorem \ref{point-count2}]
Let $\#Z_\lambda(\mathbb{F}_q)=\#\{(x_1, x_2)\in \mathbb{F}_q^2: x_1^d+x_2^d=d\lambda x_1x_2^{d-1}\}$
denote the number of $\mathbb{F}_q$-points on the $0$-dimensional variety $Z_\lambda^d: x_1^d+x_2^d=d\lambda x_1x_2^{d-1}$.
If $N_{\mathbb{F}_q}(Z_{\lambda})$ denotes the number of points on $Z_{\lambda}$ in
$\mathbb{P}_{\mathbb{F}_q}^{1}$ then
\begin{eqnarray}\label{eq-36}
N_{\mathbb{F}_q}(Z_{\lambda})=\frac{\#Z_\lambda(\mathbb{F}_q)-1}{q-1}.
\end{eqnarray}
From the proof of Theorem \ref{point-count} we have
\begin{eqnarray}\label{eq-32}
q\cdot\#Z_\lambda(\mathbb{F}_q)&=&q^2+q-1+B+A,
\end{eqnarray}
where $B=2\sum_{z,x_1\in\mathbb{F}_q^{\times}}\theta(zx_1^d)$ and $A=\sum_{z,x_1,x_2\in\mathbb{F}_q^{\times}}\theta(zx_1^d)\theta(zx_2^d)
\theta(-d\lambda zx_1x_2^{d-1}).$ Using Lemma \ref{lemma1} and Lemma \ref{lemma2} we obtain $B=-2(q-1)$.
Also, proceeding similarly as shown in the proof of Theorem \ref{point-count} we have
\begin{eqnarray}\label{eq-30}
A&=&\sum_{l=0}^{q-2}g(T^{-l})g(T^{-(d-1)l})g(T^{dl})T^{-dl}(-d\lambda).
\end{eqnarray}
Here $d\geq3$ is odd. From Lemma \ref{lemma9}, we have
\begin{eqnarray}
g(T^{dl})&=&\frac{\prod_{i=0}^{d-1}g(T^{l+\frac{i(q-1)}{d}})}{q^{\frac{d-1}{2}}T^{\frac{(d-1)(d+1)(q-1)}{8d}}(-1)}T^{dl}(d),\nonumber\\
g(T^{-(d-1)l})&=&\frac{\prod_{i=0}^{d-2}g(T^{-l-\frac{i(q-1)}{d-1}})}{q^{\frac{d-3}{2}}g(\phi)T^{\frac{(d-3)(q-1)}{2}}(-1)}
T^{-(d-1)l}(d-1).\nonumber
\end{eqnarray}
Plugging these two expressions in \eqref{eq-30} we deduce that
\begin{eqnarray}
A&=&\frac{T^{\frac{(3d-1)(q-1)}{8d}}(-1)}{q^{d-2}g(\phi)}\sum_{l=0}^{q-2}g(T^{-l})\prod_{i=0}^{d-1}g(T^{l+\frac{i(q-1)}{d}})
\prod_{i=0}^{d-2}g(T^{-l-\frac{i(q-1)}{d-1}})T^l\left(-\frac{1}{\alpha}\right)\nonumber\\
&=&\frac{T^{\frac{(3d-1)(q-1)}{8d}}(-1)}{q^{d-2}g(\phi)}\sum_{l=0}^{q-2}\{g(T^{l})g(T^{-l-\frac{q-1}{2}})\}
g(T^{-l})^2\prod_{i=1}^{d-1}g(T^{l+\frac{i(q-1)}{d}})\nonumber\\
&\times&\prod_{\substack{i=1\\ i\neq\frac{d-1}{2}}}^{d-2}g(T^{-l-\frac{i(q-1)}{d-1}})T^l\left(-\frac{1}{\alpha}\right),\nonumber
\end{eqnarray}
where $\alpha=\lambda^d(d-1)^{d-1}$.
\par
Now, pairing the terms under summation we obtain
\begin{eqnarray}
A&=&\frac{T^{\frac{(3d-1)(q-1)}{8d}}(-1)}{q^{d-2}g(\phi)}\sum_{l=0}^{q-2}\{g(T^{l})g(T^{-l-\frac{q-1}{2}})\}
g(T^{-l})^2\prod_{i=1}^{d-1}g(T^{l+\frac{i(q-1)}{d}})\nonumber\\
&\times&\prod_{\substack{i=1\\ i\neq\frac{d-1}{2}}}^{d-2}g(T^{-l-\frac{i(q-1)}{d-1}})T^l\left(-\frac{1}{\alpha}\right)\nonumber\\
&=&\frac{T^{\frac{(3d-1)(q-1)}{8d}}(-1)}{q^{d-2}g(\phi)}\sum_{l=0}^{q-2}T^{l}\left(-\frac{1}{\alpha}\right)
\{g(T^{l})g(T^{-l-\frac{q-1}{2}})\}
\{g(T^{l+\frac{q-1}{d}})g(T^{-l-\frac{q-1}{d-1}})\}\nonumber\\
&\times&\{g(T^{l+\frac{2(q-1)}{d}})g(T^{-l-\frac{2(q-1)}{d-1}})\}\cdots
\{g(T^{l+(\frac{d-1}{2}-1)\frac{(q-1)}{d}})g(T^{-l-(\frac{d-1}{2}-1)\frac{(q-1)}{d-1}})\}\nonumber\\
&\times&\{g(T^{l+(\frac{d-1}{2})\frac{(q-1)}{d}})g(T^{-l})\}\{g(T^{l+(\frac{d+1}{2})\frac{(q-1)}{d}})g(T^{-l})\}\nonumber\\
&\times&\{g(T^{l+(\frac{d+3}{2})\frac{(q-1)}{d}})g(T^{-l-(\frac{d+1}{2})\frac{(q-1)}{d-1}})\}\cdots
\{g(T^{l+\frac{(d-1)(q-1)}{d}})g(T^{-l-\frac{(d-2)(q-2)}{d-1}})\}.\nonumber
\end{eqnarray}
Applying Lemma \ref{lemma15} and Lemma \ref{fusi3} we deduce that
\begin{eqnarray}\label{eq-34}
A&=&\frac{q^{\frac{d+1}{2}}}{g(\phi)}\sum_{l=0}^{q-2}T^{l}\left(-\frac{1}{\alpha}\right)\{g(T^{l})g(T^{-l-\frac{q-1}{2}})\}
\left(\begin{array}{c}
                T^{l+\frac{q-1}{d}} \\
                T^{l+\frac{q-1}{d-1}}
              \end{array}
\right)\nonumber\\
&\times&\left(\begin{array}{c}
                T^{l+\frac{2(q-1)}{d}} \\
                T^{l+\frac{2(q-1)}{d-1}}
              \end{array}
\right)\cdots\left(\begin{array}{c}
                T^{l+(\frac{d-3}{2})\frac{(q-1)}{d}} \\
                T^{l+(\frac{d-3}{2})\frac{(q-1)}{d-1}}
              \end{array}
\right)\left(\begin{array}{c}
                T^{l+(\frac{d-1}{2})\frac{(q-1)}{d}} \\
                T^{l}
              \end{array}
\right)\nonumber\\
&\times&\left(\begin{array}{c}
                T^{l+(\frac{d+1}{2})\frac{(q-1)}{d}} \\
                T^{l}
              \end{array}
\right)\left(\begin{array}{c}
                T^{l+(\frac{d+3}{2})\frac{(q-1)}{d}} \\
                T^{l+(\frac{d+1}{2})\frac{(q-1)}{d-1}}
              \end{array}
\right)\cdots\left(\begin{array}{c}
                T^{l+\frac{(d-1)(q-1)}{d}} \\
                T^{l+\frac{(d-2)(q-1)}{d-1}}
              \end{array}
\right)\nonumber\\
&=&q^{\frac{d+1}{2}}\sum_{l=0}^{q-2}T^{l}\left(-\frac{1}{\alpha}\right)
\left\{\frac{g(T^{l})g(T^{-l-\frac{q-1}{2}})}{g(\phi)}\right\}
\left(\begin{array}{c}
                T^{l+\frac{q-1}{d}} \\
                T^{l+\frac{q-1}{d-1}}
              \end{array}
\right)\nonumber\\
&\times&\left(\begin{array}{c}
                T^{l+\frac{2(q-1)}{d}} \\
                T^{l+\frac{2(q-1)}{d-1}}
              \end{array}
\right)\cdots\left(\begin{array}{c}
                T^{l+(\frac{d-3}{2})\frac{(q-1)}{d}} \\
                T^{l+(\frac{d-3}{2})\frac{(q-1)}{d-1}}
              \end{array}
\right)\left(\begin{array}{c}
                T^{l+(\frac{d-1}{2})\frac{(q-1)}{d}} \\
                T^{l}
              \end{array}
\right)\nonumber\\
&\times&\left(\begin{array}{c}
                T^{l+(\frac{d+1}{2})\frac{(q-1)}{d}} \\
                T^{l}
              \end{array}
\right)\left(\begin{array}{c}
                T^{l+(\frac{d+3}{2})\frac{(q-1)}{d}} \\
                T^{l+(\frac{d+1}{2})\frac{(q-1)}{d-1}}
              \end{array}
\right)
\cdots\left(\begin{array}{c}
                T^{l+\frac{(d-1)(q-1)}{d}} \\
                T^{l+\frac{(d-2)(q-1)}{d-1}}
              \end{array}
\right).
\end{eqnarray}
Lemma \ref{lemma10} yields
\begin{eqnarray}
A&=&q^{\frac{d+1}{2}}\sum_{t\in\mathbb{F}_q^{\times}}\phi(t(t-1))\sum_{l=0}^{q-2}T^l\left(\frac{1}{t\alpha}\right)
\left(\begin{array}{c}
                T^{l+\frac{q-1}{d}} \\
                T^{l+\frac{q-1}{d-1}}
              \end{array}
\right)\nonumber\\
&\times&\left(\begin{array}{c}
                T^{l+\frac{2(q-1)}{d}} \\
                T^{l+\frac{2(q-1)}{d-1}}
              \end{array}
\right)\cdots\left(\begin{array}{c}
                T^{l+(\frac{d-3}{2})\frac{(q-1)}{d}} \\
                T^{l+(\frac{d-3}{2})\frac{(q-1)}{d-1}}
              \end{array}
\right)\left(\begin{array}{c}
                T^{l+(\frac{d-1}{2})\frac{(q-1)}{d}} \\
                T^{l}
              \end{array}
\right)\nonumber\\
&\times&\left(\begin{array}{c}
                T^{l+(\frac{d+1}{2})\frac{(q-1)}{d}} \\
                T^{l}
              \end{array}
\right)\left(\begin{array}{c}
                T^{l+(\frac{d+3}{2})\frac{(q-1)}{d}} \\
                T^{l+(\frac{d+1}{2})\frac{(q-1)}{d-1}}
              \end{array}
\right)
\cdots\left(\begin{array}{c}
                T^{l+\frac{(d-1)(q-1)}{d}} \\
                T^{l+\frac{(d-2)(q-1)}{d-1}}
              \end{array}
\right)\nonumber\\
&=&q^{\frac{d+1}{2}}\sum_{t\in\mathbb{F}_q^{\times}}\phi(t(t-1))\sum_{l=0}^{q-2}T^l\left(\frac{1}{t\alpha}\right)
\left(\begin{array}{c}
                T^{l+(\frac{d-1}{2})\frac{(q-1)}{d}} \\
                T^{l}
              \end{array}
\right)
\left(\begin{array}{c}
                T^{l+\frac{q-1}{d}} \\
                T^{l+\frac{q-1}{d-1}}
              \end{array}
\right)\nonumber\\
&\times&\left(\begin{array}{c}
                T^{l+\frac{2(q-1)}{d}} \\
                T^{l+\frac{2(q-1)}{d-1}}
              \end{array}
\right)\cdots\left(\begin{array}{c}
                T^{l+(\frac{d-3}{2})\frac{(q-1)}{d}} \\
                T^{l+(\frac{d-3}{2})\frac{(q-1)}{d-1}}
              \end{array}
\right)
\left(\begin{array}{c}
                T^{l+(\frac{d+1}{2})\frac{(q-1)}{d}} \\
                T^{l}
              \end{array}
\right)\nonumber\\
&\times&\left(\begin{array}{c}
                T^{l+(\frac{d+3}{2})\frac{(q-1)}{d}} \\
                T^{l+(\frac{d+1}{2})\frac{(q-1)}{d-1}}
              \end{array}
\right)
\cdots\left(\begin{array}{c}
                T^{l+\frac{(d-1)(q-1)}{d}} \\
                T^{l+\frac{(d-2)(q-1)}{d-1}}
              \end{array}
\right)\nonumber\\
&=&q^{\frac{d-1}{2}}(q-1)\sum_{t\in\mathbb{F}_q^{\times}}\phi(t(t-1))\nonumber\\
&\times&{_{d-1}F}_{d-2}\left(\begin{array}{cccccccc}
                       \chi^{\frac{d-1}{2}}, \hspace{-.2cm}& \chi, \hspace{-.2cm}& \ldots, \hspace{-.2cm}& \chi^{\frac{d-3}{2}}, \hspace{-.2cm}
                       & \chi^{\frac{d+1}{2}},
\hspace{-.2cm} & \chi^{\frac{d+3}{2}}, \hspace{-.2cm} & \ldots, \hspace{-.2cm} & \chi^{d-1} \\
                      ~ & \psi, \hspace{-.2cm} & \ldots, \hspace{-.2cm} & \psi^{\frac{d-3}{2}},\hspace{-.2cm} & \varepsilon,
                      \hspace{-.2cm} & \psi^{\frac{d+1}{2}}, \hspace{-.2cm} & \ldots, \hspace{-.2cm}& \psi^{d-2}
\end{array}|\frac{1}{t\alpha}
\right).\nonumber
\end{eqnarray}
Now, substituting the values of $A$ and $B$ in \eqref{eq-32}, and then using \eqref{eq-36} we deduce that
\begin{eqnarray}\label{eq-37}
&&q\cdot N_{\mathbb{F}_q}(Z_\lambda)=q-1+q^{\frac{d-1}{2}}\sum_{t\in\mathbb{F}_q^{\times}}\phi(t(t-1))\nonumber\\
&&\times{_{d-1}F}_{d-2}\left(\begin{array}{cccccccc}
                       \chi^{\frac{d-1}{2}}, \hspace{-.2cm} & \chi, \hspace{-.2cm} & \ldots, \hspace{-.2cm} & \chi^{\frac{d-3}{2}}, \hspace{-.2cm}
                       & \chi^{\frac{d+1}{2}}, \hspace{-.2cm} & \chi^{\frac{d+3}{2}}, \hspace{-.2cm} & \ldots, \hspace{-.2cm} & \chi^{d-1} \\
                      ~ & \psi, \hspace{-.2cm} & \ldots, \hspace{-.2cm} & \psi^{\frac{d-3}{2}}, \hspace{-.2cm} & \varepsilon, \hspace{-.2cm}
                      & \psi^{\frac{d+1}{2}}, \hspace{-.2cm} & \ldots, \hspace{-.2cm} & \psi^{d-2}
\end{array}|\frac{1}{t\alpha}\right).
\end{eqnarray}
Finally, replacing $t$ by $\frac{1}{t}$ in \eqref{eq-37} we derive the required result.
\end{proof}
\begin{proof}[Proof of Theorem \ref{MT-4}]
For $\lambda\in\mathbb{F}_q^{\times}$, we consider
\begin{eqnarray}\label{eq-612}
A_{\lambda}=\sum_{l=0}^{q-2}g(T^{-l})g(T^{-(d-1)l})g(T^{dl})T^{l}\left(\frac{-(d-1)^{d-1}\lambda}{d^d}\right).
\end{eqnarray}
We observe that \eqref{eq-30} and \eqref{eq-612} contain the same Gauss sums.
Therefore, proceeding similarly as shown in the proof of Theorem \ref{point-count2} we deduce that
\begin{eqnarray}\label{eq-613}
A_{\lambda}&=&q^{\frac{d-1}{2}}(q-1)\sum_{t\in\mathbb{F}_q^{\times}}\phi(t(t-1))\nonumber\\
&&\times{_{d-1}F}_{d-2}\left(\begin{array}{cccccccc}
                       \chi^{\frac{d-1}{2}}, \hspace{-.2cm} & \chi, \hspace{-.2cm} & \ldots, \hspace{-.2cm} & \chi^{\frac{d-3}{2}}, \hspace{-.2cm}
                       & \chi^{\frac{d+1}{2}}, \hspace{-.2cm} & \chi^{\frac{d+3}{2}}, \hspace{-.2cm} & \ldots, \hspace{-.2cm} & \chi^{d-1} \\
                      ~ & \psi, \hspace{-.2cm} & \ldots, \hspace{-.2cm} & \psi^{\frac{d-3}{2}}, \hspace{-.2cm} & \varepsilon, \hspace{-.2cm}
                      & \psi^{\frac{d+1}{2}}, \hspace{-.2cm} & \ldots, \hspace{-.2cm} & \psi^{d-2}
\end{array}|\frac{\lambda}{t}\right),
\end{eqnarray}
where $\chi$ and $\psi$ are characters of order $d$ and $d-1$, respectively.
\par
In \cite{BS-RAMA} we express the number of distinct zeros of the polynomial $x^d+ax^i+b$ over a finite field $\mathbb{F}_q$
in terms of Greene's hypergeometric function under the condition that $i|d$ and $q\equiv1\pmod{\frac{d(d-i)}{i^2}}$.
In \cite[Eqn 17]{BS-RAMA}, we consider the following term.
\begin{eqnarray}\label{eq-612-new}
B=\frac{1}{q-1}\sum_{l=0}^{q-2}g(T^{-l})g(T^{\frac{ld}{i}})g(T^{-l(\frac{d}{i}-1)})T^{l}\left(\frac{b^{\frac{d}{i}-1}}{a^{\frac{d}{i}}} \right).
\end{eqnarray}
When $i=1$, the Gauss sums present in \eqref{eq-612} and \eqref{eq-612-new} are the same.
Therefore, proceeding similarly as shown in the proof of \cite[Thm. 1.3]{BS-RAMA} for $i=1$ we deduce that
\begin{eqnarray}\label{eq-614}
A_{\lambda}&=&q-1-\phi(-\lambda)(q-1)+(q-1)q^{\frac{d+1}{2}}\phi(-1)\nonumber\\
&&\times{_dF}_{d-1}\left(\begin{array}{ccccccc}
                   \phi, \hspace{-.2cm} & \chi, \hspace{-.2cm} & \ldots, \hspace{-.2cm} & \chi^{\frac{d-1}{2}}, \hspace{-.2cm}
                   & \chi^{\frac{d+1}{2}}, \hspace{-.2cm} & \ldots, \hspace{-.2cm} & \chi^{d-1} \\
                   ~ & \psi, \hspace{-.2cm} & \ldots, \hspace{-.2cm} & \psi^{\frac{d-1}{2}}, \hspace{-.2cm} & \psi^{\frac{d-1}{2}}, \hspace{-.2cm}
                   & \ldots, \hspace{-.2cm} & \psi^{d-2}
                 \end{array}|\lambda
\right),
\end{eqnarray}
where $\chi$ and $\psi$ are characters of order $d$ and $d-1$, respectively.
Finally, combining \eqref{eq-613} and \eqref{eq-614}, and then replacing $1/t$ by $t$ we deduce the desired summation identity.
This completes the proof of the theorem.
\end{proof}


\begin{thebibliography}{999}

\bibitem{BRS}
R. Barman, H. Rahman and N. Saikia, {\it Counting points on Dwork hypersurfaces and $p$-adic hypergeometric function},
Bull. Aust. Math. Soc., DOI: 10.1017/S0004972715001847.

\bibitem{BS5}
R. Barman, N. Saikia and D. McCarthy, {\it Summation identities and special values of hypergeometric series in the $p$-adic setting},
J. Number Theory 153 (2015), 63--84.

\bibitem{BS-FFA}
R. Barman and N. Saikia, {\it $p$-adic Gamma function and the polynomials $x^d+ax+b$ and $x^d+ax^{d-1}+b$ over $\mathbb{F}_{q}$},
Finite Fields Appl. 29 (2014), 89--105.

\bibitem{BS-RAMA}
R. Barman and N. Saikia, {\it On the polynomials $x^d+ax^i+b$ and $x^d+ax^{d-i}+b$ over $\mathbb{F}_q$
and Gaussian hypergeometric series}, Ramanujan J. 35 (2014), no. 3, 427--441.

\bibitem{evans}
B. Berndt, R. Evans, and K. Williams, {\it Gauss and Jacobi Sums}, Canadian Mathematical Society Series of Monographs and Advanced Texts,
A Wiley-Interscience Publication, John Wiley \& Sons, Inc., New York, 1998.


\bibitem{goodson}
H. Goodson, {\it Hypergeometric functions and relations to Dwork hypersurfaces}, Int. J. Number Theory, DOI: 10.1142/S1793042117500269.

\bibitem{Fuselier} J. Fuselier, \textit{Hypergeometric functions over $\mathbb{F}_p$ and relations to elliptic curve and modular forms},
Proc. Amer. Math. Soc. 138 (2010), 109--123.

\bibitem{Fuselier-McCarthy} J. Fuselier and D. McCarthy, {\it Hypergeometric type identities in the $p$-adic setting and modular forms},
Proc. Amer. Math. Soc. 144 (2016), 1493--1508

\bibitem{greene}
J. Greene, {\it Hypergeometric functions over finite fields}, Trans. Amer. Math. Soc. 301 (1987), no. 1, 77--101.

\bibitem{gross}
B. H. Gross and N. Koblitz, {\it Gauss sum and the $p$-adic $\Gamma$-function}, Annals of Mathematics 109 (1979), 569--581.

\bibitem{ireland}
K. Ireland and M. Rosen, {\it A Classical Introduction to Modern Number Theory}, Springer International Edition, Springer, 2005.

\bibitem{katz} N. M. Katz, {\it Exponential Sums and Differential Equations}, Princeton University Press, Princeton, NJ, 1990.

\bibitem{kob} N. Koblitz, {\it $p$-adic analysis: a short course on recent work}, London Math. Soc. Lecture
Note Series, 46. Cambridge University Press, Cambridge-New York, 1980.

\bibitem {Lang} S. Lang, \textit{Cyclotomic Fields I and II},
Graduate Texts in Mathematics, vol. 121, Springer-Verlag, New York, 1990.




\bibitem{mccarthy3}
D. McCarthy, {\it Extending Gaussian hypergeometric series to the $p$-adic setting}, Int. J. Number Theory 8 (2012), no. 7, 1581--1612.

\bibitem{mccarthy2}
D. McCarthy, {\it The trace of Frobenius of elliptic curves and the $p$-adic gamma function}, Pacific J. Math. 261 (2013), no. 1, 219--236.


\bibitem{salerno} A. Salerno, \textit{Counting points over finite fields and hypergeometric functions},
Funct. Approx. Comment. Math. 49 (2013), no. 1, 137--157.

\end{thebibliography}
\end{document}